\documentclass[11pt]{amsart}
\usepackage[english]{babel}
\usepackage[T1]{fontenc}
\usepackage[ansinew]{inputenc}
\usepackage{amsmath}
\usepackage{amsfonts}
\usepackage{ mathrsfs }
\usepackage{amssymb}
\usepackage{textcomp}
\usepackage{enumitem}
\usepackage[hidelinks]{hyperref}
\usepackage[arrow, matrix, curve]{xy}
\usepackage{comment}
\usepackage{color}

\numberwithin{paragraph}{section}
\numberwithin{equation}{section}

\newtheorem{satz}{Theorem}[section]

\newtheorem{lem}[satz]{Lemma}

\newtheorem{prop}[satz]{Proposition}

\newtheorem{kor}[satz]{Corollary}
\theoremstyle{definition}
\newtheorem{defn}[satz]{Definition}
\newtheorem{bem}[satz]{Remark}
\newtheorem{Ex}[satz]{Example}

\newtheorem{question}[satz]{Question}

\newcommand{\Z}{\mathbb{Z}}

\newcommand{\Q}{\mathbb{Q}}
\newcommand{\R}{\mathbb{R}}

\newcommand{\C}{\mathbb{C}}
\newcommand{\T}{\mathbb{T}}

\newcommand{\G}{\mathbb{G}}

\newcommand{\Xcal}{\mathcal{X}}

\newcommand{\Ocal}{\mathcal{O}}

\newcommand{\Linear}{\mathbb{L}}

\newcommand{\del}{\partial}
\newcommand{\Xan}{X^{\an}}
\newcommand{\inj}{\hookrightarrow}

	\DeclareMathOperator{\an}{an}
	\DeclareMathOperator{\PD}{PD}

	\DeclareMathOperator{\trop}{trop}

	\DeclareMathOperator{\Hom}{Hom}
	
	\DeclareMathOperator{\Spec}{Spec}

	\DeclareMathOperator{\Trop}{Trop}

	\DeclareMathOperator{\im}{im}
	\DeclareMathOperator{\lin}{lin}

	\DeclareMathOperator{\supp}{supp}

	\DeclareMathOperator{\Pic}{Pic}

	\DeclareMathOperator{\Monoids}{Monoids}
	
	\DeclareMathOperator{\CH}{CH}
	\DeclareMathOperator{\tor}{\mathcal{T}}
	\DeclareMathOperator{\map}{map}

	\DeclareMathOperator{\can}{can}
	
	\DeclareMathOperator{\HH}{H}
	\DeclareMathOperator{\dR}{dR}
	\DeclareMathOperator{\CLD}{CLD}
	
	\DeclareMathOperator{\Cund}{\underline{C}}
	\DeclareMathOperator{\Fbb}{\mathbb{F}}
	\DeclareMathOperator{\Smooth}{Smooth}
	\DeclareMathOperator{\Hodge}{Hodge}
	\DeclareMathOperator{\cyc}{cyc}
	\DeclareMathOperator{\glob}{global}

\DeclareMathOperator{\AS}{\mathcal{A}}

\DeclareMathOperator{\CS}{\mathcal{C}}

\DeclareMathOperator{\FS}{\mathcal{F}}

\DeclareMathOperator{\HS}{\mathcal{H}}

\DeclareMathOperator{\Scal}{\mathcal{S}} 
\DeclareMathOperator{\TS}{\mathcal{T}}

\DeclareMathOperator{\A}{\mathbb{A}}

\def\quotient#1#2{\raise0.75ex\hbox{$\,#1$}\big/\lower0.75ex\hbox{$#2\,$}}

\title{Tropical cohomology with integral coefficients for analytic spaces}

\author[P.~Jell]{Philipp Jell}
\address{P. Jell, Fakult\"at Mathematik, Universit{\"a}t 
Regensburg, 93040 Regensburg, Germany}
\email{philipp.jell@mathematik.uni-regensburg.de}

\thanks{The author was supported by the DFG Collaborative Research Center 1085 ``Higher Invariants''.}

\begin{document}
\begin{abstract}
We study tropical Dolbeault cohomology for Berkovich analytic spaces, as defined by Chambert-Loir and Ducros. 
We provide a construction that lets us pull back classes in tropical cohomology 
to classes in tropical Dolbeault cohomology as well as check whether those classes are non-trivial. 
We further define tropical cohomology with integral coefficients on the Berkovich space and 
provide some computations. 
Our main tool is extended tropicalization of toric varieties as introduced by Kajiwara and Payne. 

\bigskip


\bigskip

\end{abstract}

\maketitle 

\section{Introduction}

Real valued-differential forms and currents on Berkovich analytic spaces 
were introduced by Chambert--Loir and Ducros 
in their fundamental preprint \cite{CLD}. 
They provide a notion of bigraded differential forms and currents 
on these spaces that has striking similarities with the complex 
of smooth differential forms on complex analytic spaces. 
The definition works by formally pulling back Lagerberg's superforms on $\R^n$ 
along tropicalization maps. 
These tropicalization maps are induced by mapping open subsets of the analytic space
to analytic tori and then composing with the tropicalization maps of the tori. 

Payne, and independently Kajiwara \cite{Payne, Kajiwara}, generalized this tropicalization
procedure from tori to general toric varieties and 
Payne showed that the Berkovich analytic space is the inverse limit 
over all these tropicalizations. 

Shortly after the preprint by Chambert--Loir and Ducros, Gubler 
showed that one may, instead of considering arbitrary analytic maps to tori, 
restrict one's attention to algebraic closed embeddings 
if the analytic space is the Berkovich analytification of an algebraic variety \cite{Gubler}. 

Let $K$ be a field that is complete with respect to a non-archimedean absolute value
and let $X$ be a variety over $K$. 
We write $\Gamma = \log \vert K^* \vert$ for the value group of $K$
and $\Xan$ for the Berkovich analytification of $X$. 
Both the approach by Gubler and the one by Chambert--Loir and Ducros 
provide bigraded complexes of sheaves of differential forms $(\AS^{\bullet, \bullet}, d', d'')$ on $\Xan$. 
We denote by $\HH^{*,*}$ (resp.~$\HH^{*,*}_c$) the cohomology of the complex of global sections 
(resp.~global sections with compact support) with respect to $d''$. 
 
In this paper, we generalize Gubler's approach, showing that one can define 
forms on Berkovich analytic spaces by using certain classes of embeddings of open subsets  
into toric varieties.
Given a fine enough family of tropicalizations $\Scal$ (see Section \ref{families of trop} for the definition
of this notion) 
we obtain a bigraded complex of sheaves $(\AS^{\bullet, \bullet}_{\Scal}, d', d'')$ on $\Xan$. 
We show that for many useful $\Scal$, our complex $\AS^{\bullet, \bullet}_{\Scal}$ is
isomorphic to $\AS^{\bullet, \bullet}$.

For the rest of the introduction, we make the very mild assumption
the $X$ is normal and admits at least one closed embedding into a toric variety. 

The general philosophy of this paper and also the definition
of forms by Chambert-Loir and Ducros and Gubler
is that we can transport constructions done for tropical varieties to Berkovich spaces
by locally pulling back along tropicalizations. 
While Chambert-Loir, Ducros and Gubler used only tropicalization maps of tori, 
we will also allow tropicalization maps of general toric varieties. 
We will show advantages of this equivalent approach throughout the paper. 

The definitions by both Chambert-Loir and Ducros as well as Gubler work with local embeddings.
We show that we can also work with global embeddings. 
Our constructions provides us with the following: 
Let $\varphi \colon X \to Y_{\Sigma}$ be a closed embedding 
into a toric variety. 
Then we obtain pullback morphisms
\begin{align}
\label{eq tropstar1} \trop^* &\colon \HH^{p,q}(\Trop(\varphi(X))) \to \HH^{p,q}(\Xan) \text{ and } \\
\label{eq tropstar2} \trop^*&\colon \HH^{p,q}_c(\Trop(\varphi(X))) \to \HH^{p,q}_c(\Xan)
\end{align}
in cohomology. 

Note that (\ref{eq tropstar1}) and (\ref{eq tropstar2}) were not in general available 
in the approaches by Chambert--Loir and Ducros resp.~Gubler. 
This construction allows us to explicitly construct classes in tropical Dolbeault cohomology. 

For cohomology with compact support, we even obtain all classes this way:
\begin{satz} [Theorem \ref{thm cohomology limit}] \label{thm intro 1}
We have 
\begin{align*}
\HH^{p,q}_c(\Xan) = \varinjlim_{\varphi \colon X \to Y_\Sigma} \HH_c^{p,q}(\Trop(\varphi(X)).
\end{align*}
where the limit runs over all closed embeddings of $X$ into toric varieties. 
\end{satz}

We also show that the analogous result for $\HH^{p,q}$ is not true (Remark \ref{bem non-compact non-true}). 

Further, in certain cases, we can check whether one of these classes is non-trivial on
the tropical side. 

\begin{satz}  [Theorem \ref{thm trop injective}]
Assume that $\Trop(\varphi(X))$ is smooth. 
Then (\ref{eq tropstar1}) and (\ref{eq tropstar2}) are both injective. 
\end{satz}

We exhibit three examples in Section \ref{sect non-trivial classes}, 
namely Mumford curve, curves of good reduction and toric varieties. 

Another construction that we transport over from the tropical to the analytic world is 
cohomology with coefficients other than the real numbers. 
For a subring $R$ of $\R$, we define a cohomology theory
\begin{align*}
\HH^{*,*}_{\trop}(\Xan, R) \text{ and } 
\HH^{*,*}_{\trop,c}(\Xan, R)
\end{align*}
with values in $R$-modules. 
Liu introduced in \cite{Liu} a canonical rational subspace of $\HH^{p,q}(\Xan)$. 
We show that this space agrees with $\HH^{p,q}(\Xan, \Q)$ as defined in this paper 
(Proposition \ref{prop rational classes}). 

We obtain the analogue of Theorem \ref{thm intro 1} where on the 
right hand side we have tropical cohomology with coefficients in $R$ (Proposition \ref{prop comparison coefficients}),
and we provide an explicit isomorphism
\begin{align*}
\dR \colon \HH^{p,q}(\Xan) \to \HH^{p,q}_{\trop}(\Xan, \R).
\end{align*}
which is a version of de Rham's theorem in this context (Theorem \ref{Tropical analytic de Rham theorem}).
Liu introduced in \cite{Liu2} a \emph{monodromy operator}
\begin{align*}
M \colon \HH^{p,q}_{\trop, c}(\Xan) \to \HH^{p-1,q+1}_{\trop, c}(\Xan)
\end{align*}
that respects rational classes if $\log \vert K^* \vert \subset \Q$  \cite[Theorem 5.5 (1)]{Liu2}. 
Mikhalkin and Zharkov introduce in \cite{MikZhar} a wave operator
\begin{align*}
W \colon \HH^{p,q}_c(\Trop_{\varphi}(X), \R) \to \HH^{p-1,q+1}_c(\Trop_{\varphi}(\Xan), \R). 
\end{align*}
Note that both these operators are also available without compact support. 

We show that $W$ can be used to give an operator on $\HH^{*, *}_c(\Xan)$ 
and that this operator agrees with $M$ up to sign in Corollary \ref{kor wave mondromy analytic}. 

We also obtain the following result regarding the interaction
between the wave operator and the coefficients of the cohomology groups:

\begin{satz} 
Let $R$ be a subring of $\R$ and $R[\Gamma]$ the smallest subring of $\R$ 
that contains both $R$ and $\Gamma$. 
Then the wave operator $W$ restricts to a map
\begin{align*}
W \colon \HH^{p,q}_c(\Xan, R) \to \HH^{p-1,q+1}_c(\Xan, R[\Gamma]).
\end{align*}
\end{satz}

As $W$ and $M$ agree up to sign and Liu's subspace of rational classes agrees with $\HH^{*,*}(\Xan, \Q)$, 
this generalizes Liu's result for $\Gamma \subset R = \Q$.

We now sketch the organization of the paper. 
In Section \ref{sec toric varieties} we recall background on toric varieties and their tropicalizations. 
Section \ref{sect tropical cohomology} contains all constructions on tropical varieties 
that are needed for the paper. 
Most of these should be known to experts, 
however we still chose to list them for completeness. 
The main new result is the identification of the wave and monodromy operator, 
which is based on Lemma \ref{thm monodromy wave trop}. 
In Section \ref{families of trop} we consider what we call \emph{families of tropicalizations},
which is what we will use to define forms on Berkovich spaces. 
We give the definitions and some examples of families that we will consider. 
In Section \ref{sect superforms} we define for a fine enough family of tropicalizations $\Scal$ 
a bigraded complex $\AS^{\bullet, \bullet}_{\Scal}$ of sheaves of differential forms on Berkovich spaces.
We also provide some conditions under which those complexes are isomorphic 
for different $\Scal$
and introduce tropical Dolbeault cohomology for $\Xan$. 
In Section \ref{sect comparison theorems}, we prove that for so-called admissible families $\Scal$, 
the complexes $\AS^{\bullet, \bullet}_{\Scal}$ are isomorphic to $\AS^{\bullet, \bullet}$.   
In Section \ref{sect integration} we discuss integration of top-dimensional differential
forms with compact support. 
In Section \ref{sect tropical cohomology Xan} we introduce tropical cohomology with coefficients for $\Xan$ and 
compare it with $\HH^{*,*}$. 
In Section \ref{sect wave and monodromy operator}, we discuss the relation between 
the wave and monodromy operators and consequences thereof. 
Section \ref{sect non-trivial classes} provides partial computations of $\HH^{p,q}(\Xan)$ and $\HH^{p,q}(\Xan, R)$
for curves and toric varieties, using our new approaches. 
In Section \ref{sect open questions} we list open questions that one might ask as a consequence of our results.

\section*{Acknowledgments}

The author would like to thank Walter Gubler, Johannes Rau and Kristin Shaw for helpful comments
and suggestions. 
The idea for the proof of Theorem \ref{thm monodromy wave trop} came from
joint work with Johannes Rau and Kristin Shaw on the paper \cite{JRS}. 
The author would like to express his gratitude for being allowed to use those ideas. 

Parts of this work already appeared in a more a hoc and less conceptual way in the author's 
PhD thesis \cite{JellThesis}. 

The author would also like to thank the referee for their detailed report and specific remarks  
which greatly improved the paper.

\section*{Notations and conventions}

Throughout $K$ is a field that is complete 
with respect to a (possibly trivial) non-archimedean absolute value. 
We denote its value group by $\Gamma := \log \vert K^* \vert$. 
If the absolute value is non-trivial, we normalize it in such a way that $\Z \subset \Gamma$. 
A variety $X$ is an geometrically integral separated $K$-scheme of finite type. 
For any variety $X$ over $K$, we will throughout the paper denote by 
$\Xan$ the analytification in the sense of Berkovich \cite{BerkovichSpectral}.

\section{Toric varieties and tropicalization} \label{sec toric varieties}

\subsection{Toric varieties} 

Let $N$ be a free abelian group of finite rank, $M$ its dual and 
denote by $N_\R$ resp.~$M_\R$ the respective scalar extensions to $\R$. 

\begin{defn} 
A \emph{rational cone} $\sigma \in N_\R$ is a polyhedron defined by equations of the form 
$\varphi(\,.\,) \geq 0$ with $\varphi \in M$, 
that does not contain a positive dimensional linear subspace. 
A \emph{rational fan} $\Sigma$ in $N_\R$ is a polyhedral complex all of whose polyhedra are rational cones. 
For $\sigma \in \Sigma$ we define the monoid 
\begin{align*}
S_\sigma := \{ \varphi \in M \mid \; \varphi(v) \geq 0 \text{ for all } v \in \sigma \}.
\end{align*}
We denote by $U_\sigma := \Spec(K[ S_\sigma])$. 
For $\tau \prec \sigma$ we obtain an open immersion $U_\tau \rightarrow U_\sigma$. 
We define the toric variety $Y_\Sigma$ to be the gluing of the 
$(U_\sigma)_{\sigma \in \Sigma}$ along these open immersions.  
For an introduction to toric varieties, see for example  \cite{Fulton}.
\end{defn}

\begin{bem} \index{Toric variety}
The toric variety $Y_\Sigma$ comes with an open immersion $T \rightarrow Y_\Sigma$, 
where $T = \Spec (K[M])$ and 
a $T$-action that extends the group action of $T$ on itself by translation. 
In fact any normal variety with such an immersion and action arises by the above described procedure 
(\cite[Corollary 3.1.8]{CLS}). 
This was shown by Sumihiro. 

Choosing a basis of $N$ gives an identification $N \simeq \Z^r \simeq M$ and $T \simeq \G_m^r$. 
\end{bem}

\begin{defn}
A map $\psi \colon Y_{\Sigma} \to Y_{\Sigma'}$ is called a \emph{morphism of toric varieties} 
if it is equivariant with respect to the torus actions and restricts to a morphism 
of algebraic groups on dense tori. 
It is called an \emph{affine map of toric varieties} 
if it is a morphism of toric varieties composed with a multiplicative torus translation.
\end{defn}

\begin{bem}
A morphism of toric varieties $\psi \colon Y_{\Sigma} \to Y_{\Sigma'}$ 
is induced by a morphism of corresponding fans, 
meaning a linear map $N \to N'$ that maps cones in $\Sigma$ to cones in $\Sigma'$. 
\end{bem}

\subsection{Tropical toric varieties}

Let $\Sigma$ be a rational fan in $N_\R$. 
We write $\T := \R \cup \{-\infty\}$.
For $\sigma \in \Sigma$ we define $N(\sigma) := N_\R / \langle \sigma \rangle_\R$. 
We write
\begin{align*}
N_\Sigma =  \coprod \limits_{\sigma \in \Sigma} N(\sigma).
\end{align*}
We call the $N(\sigma)$ the \emph{strata} of $N_\Sigma$. 
Note that $N_{\Sigma}$ has a canonical action by $N$ and $N_\R$
and the strata are the strata of the action of $N_\R$. 
We endow $N_\Sigma$ with a topology in the following way:

For $\sigma \in \Sigma$ write $N_\sigma = \coprod \limits_{\tau \prec \sigma} N(\tau)$. 
This is naturally identified with $\Hom_{\Monoids}(S_\sigma, \T)$. 
We equip $\T^{S_\sigma}$ with the product topology and give $N_\sigma$ the subspace topology.
For $\tau \prec \sigma$, the space $\Hom(S_\tau, \T)$ is naturally identified with the open
subspace of $\Hom_{\Monoids}(S_\sigma, \T)$ of maps that map $\tau^{\perp} \cap M$ to $\R$. 
We define the topology of $N_\Sigma$ to be the one obtained by gluing along these identifications.

\begin{defn}
We call the space $N_\Sigma$ a  \emph{tropical toric variety}.
\end{defn}

We would like to remark that the tropical toric variety $N_\Sigma$
can also be constructed by glueing its affine pieces along monomial maps \cite[Section 3.2]{MikRau}

Note that $N_\Sigma$ contains $N_\R$ as a dense open subset. 
For a subgroup $\Gamma$ of $\R$ and each stratum $N(\sigma)$ 
we call the set $N(\sigma)_\Gamma := (N \otimes \Gamma) / \langle \sigma \rangle_\Gamma$ 
the \emph{set of $\Gamma$-points}.   

Let $\Sigma$ and $\Sigma'$ be fans in $N_\R$ and $N'_\R$ respectively. 
Let $L \colon N \to N'$ be a linear map such that $L_\R$ maps every cone 
in $\Sigma$ into a cone in $\Sigma'$. 
Such a map canonically induces a map $N_\Sigma \to N_{\Sigma'}$ 
that is continuous and linear on each stratum. 

\begin{defn}
A map $N_\Sigma \to N_{\Sigma'}$ that 
arises this way is called \emph{morphism of tropical toric varieties}. 

An \emph{affine map} of tropical toric varieties is a map that 
is the composition of morphism of toric varieties with an $N_\R$-translation. 
\end{defn}

\subsection{Tropicalization}

Let $\Sigma$ be a rational fan in $N_\R$. 
Denote by $Y_\Sigma$ the associated toric variety and by $N_\Sigma$ the associated 
tropical toric variety.

\begin{defn} 
Payne defined in \cite{Payne} a tropicalization map 
\begin{align*}
\trop_{\Sigma} \colon Y_\Sigma^{\an} \rightarrow N_\Sigma
\end{align*}
to the topological space $N_\Sigma$ as follows: 
For $\tau \prec \sigma$, the space $\Hom(S_\tau, \T)$ is naturally identified with the open
subspace of $\Hom_{\Monoids}(S_\sigma, \T)$ of maps which map $\tau^{\perp} \cap M$ to $\R$. 
The map $\trop \colon U_\sigma^{\an} \to \Trop(U_\sigma)$ is then defined by mapping 
$\vert\,.\,\vert_x \in U_\sigma^{\an}$ to the homomorphism 
$u \mapsto \log \vert u\vert_x \in \Trop(U_\sigma) = \Hom(S_\sigma, \T)$. 
We will often write $\Trop(Y_\Sigma) := N_\Sigma$. 

For $Z$ a closed subvariety of $Y_\Sigma$ we define $\Trop(Z)$ to be the image of $Z^{\an}$ under 
$\trop \colon Y_\Sigma^{\an} \rightarrow \Trop(Y_\Sigma)$.
\end{defn}

\begin{defn}
The construction $\Trop(Y_\Sigma)$ is functorial with respect to affine maps of toric varieties. 
In particular, for a morphism (resp. affine map) of toric varieties $\psi \colon Y_{\Sigma} \to Y_{\Sigma'}$, 
we obtain a morphism (resp.~affine map) of tropical toric varieties
$\Trop(\psi) \colon \Trop(Y_\Sigma) \to \Trop(Y_{\Sigma'})$. 

If $\psi$ is a closed immersion, then $\Trop(\psi)$ is a homeomorphism onto its image. 
\end{defn}

\begin{Ex} \label{trop of affine embedding}
Affine space $\A^r = \Spec K[T_1,\dots,T_r]$ is the toric variety that arises from the cone 
$\{ x \in \R^r \mid x_i \geq 0 \text{ for all } i \in [r] \}$. 
By definition the tropicalization in then $\T^r$ and the map 
$\trop \colon \A^{r, \an} \to \T^r; \; \vert\;.\;\vert \mapsto (\log\vert T_i \vert)_{i \in [r]}$. 

Let $\sigma$ be a cone in $N_\R$. 
We pick a finite generating set $b_1,\dots,b_r$ of the monoid $S_\sigma$. 
Let $U_\sigma$ be the affine toric variety associated to a cone $\sigma$. 
Then we have a surjective map $K[T_1,\dots,T_r] \to K[S_{\sigma}]$, 
which induces a toric closed embedding $\varphi_B \colon U_\sigma \to \A^r$.
By functoriality of tropicalization we also get a morphism of tropical toric varieties
$\Trop(U_\sigma) \to \T^r$ that is a homeomorphism onto its image. 
\end{Ex}

\subsection{Tropical subvarieties of tropical toric varieties}

In this section we fix a subgroup $\Gamma \subset \R$.

\begin{defn}
An integral $\Gamma$-affine polyhedron in $N_\R$ is a set defined by finitely many 
inequalities of the form $\varphi(\;.\;) \geq r$ for $\varphi \in M, r \in \Gamma$. 
An integral $\Gamma$-affine polyhedron on $N_\Sigma$ is 
the topological closure of an integral $\Gamma$-affine polyhedron in $N(\sigma_\tau)$ for $\sigma_\tau \in \Sigma$. 

Let $\tau$ be an integral $\Gamma$-affine polyhedron in $N_{\Sigma}$. 
For $\sigma \prec \sigma'$ we have that $N(\sigma') \cap \tau$ is a polyhedron in $N(\sigma')$ 
(that might be empty) and we consider this as a face of $\tau$. 
Further we denote by 
$\Linear(\tau) = \{\lambda (u_1 - u_2) \mid u_1, u_2 \in \tau,\;  \lambda \in \R \} \subset N(\sigma_\tau)$
the \emph{linear space of $\tau$}. 
If $\tau$ is integral $\Gamma$-affine, $\Linear(\tau)$ contains a canonical lattice that we denote by $\Z(e)$. 
\end{defn}

\begin{defn}
A tropical subvariety of a tropical toric variety is given the support of 
a integral $\Gamma$-affine polyhedral complex with weights attached 
to its top dimensional faces, satisfying the balancing condition. 
\end{defn}

\begin{defn}
Let $Z$ be a closed subvariety of a toric variety $Y_{\Sigma}$.
Then $\Trop(Z) := \trop(Z^{\an}) \subset N_{\R}$ is a tropical subvariety of $\Trop(Y_\Sigma)$. 
For a variety $X$ and a closed embedding $\varphi \colon X \to Y_\Sigma$ we write
$\Trop_{\varphi}(X) := \Trop(\varphi(X))$ and 
$\trop_{\varphi} := \trop \circ \varphi^{\an} \colon \Xan \to \Trop_{\varphi}(X)$. 
\end{defn}
 
We will never explicitly use the weights nor the balancing condition, 
so the reader may be happy with the fact that there are weights and that they satisfy 
the balancing condition. 
If they are not happy with this, let us refer them to the excellent introduction \cite{Gubler2}. 
The reader who already knows the balancing condition for tropical varieties in $N_\R$ 
will be glad to hear that the balancing condition for $X$ is exactly the classical 
balancing condition for $X \cap N_\R$, 
there are no additional properties required at infinity.

\section{Constructions in Cohomology of tropical varieties} \label{sect tropical cohomology}

In this section, $R$ is a ring such that $\Z \subset R \subset \R$ and 
 $\Gamma$ is a subgroup of $\R$ that contains $\Z$. 
Further $N$ is a free abelian group of finite rank, $M$ is its dual, 
and $\Sigma$ is a rational fan in $N_\R$. 
Additionally $X$ is an integral $\Gamma$-affine tropical subvariety of $N_\Sigma$. 

\begin{defn} 
Let $U \subset N_\R$ an open subset. 
A \emph{superform} of bidegree $(p,q)$ is an element of 
\begin{align*}
\AS^{p,q}(U)  = C^\infty(U) \otimes \Lambda^p M \otimes \Lambda^q M
\end{align*}
\end{defn}

\begin{bem}
There are differential operators $d'$ and $d''$ and a wedge product 
which are induced by the usual differential operator and wedge 
product on differential forms.
\end{bem}

\begin{defn}
For an open subset $U$ of $N_\Sigma$ we write $U_\sigma := U \cap N_\sigma$. 
A superform of bidegree $(p,q)$ on $U$ is given by a collection $\alpha = (\alpha_\sigma)_{\sigma \in \Sigma}$
such that $\alpha_\sigma \in \AS^{p,q}(U_\sigma)$ and for each $\sigma$ and each $x \in U_\sigma$ there 
exists an open neighborhood $U_x$ of $x$ in $U$ such that for each $\tau \prec \sigma$ 
we have $\pi_{\sigma, \tau}^* \alpha_\sigma = \alpha_\tau$. 
We call this the \emph{condition of compatibility at the boundary}. 

For a polyhedron $\sigma$ in $N_\Sigma$ we can define the restriction of a superform $\alpha$ to $\sigma$. 
Let $\Omega$ be an open subset of $\vert \CS \vert$ for a polyhedral complex $\CS$ in $N_\Sigma$. 
The space of superforms of bidegree $(p,q)$ on $\Omega$ is defined as the set of pairs $(U, \alpha)$ where $U$ 
is an open subset of $N_\Sigma$ such that $U \cap \vert \CS \vert = \Omega$ and $\alpha \in \AS^{p,q}(U)$. 
Two such pairs are identified if their restrictions to $\sigma \cap \Omega$ agree for every $\sigma \in \CS$. 
\end{defn}

\begin{defn}
For a tropical subvariety $X$ of a tropical toric variety we obtain a double complex of sheaves 
$(\AS^{\bullet, \bullet}, d', d'')$ on $X$. 
We define $\HH^{p,q}(X)$ (resp.~$\HH^{p,q}_c(X)$) 
as the cohomology of the complex of global sections (resp.~global sections with compact support)
with respect to $d''$.
\end{defn}

\begin{defn}
There exists an integration map $\int_X \colon \AS^{n,n}_c(X) \to \R$ 
that satisfies Stokes' theorem and thus descends to cohomology.
\end{defn}

\begin{bem}
Superforms on tropical subvarieties of tropical toric varieties are functorial with respect to
affine maps of
tropical toric varieties \cite{JSS}.
\end{bem}

\begin{defn}
Let $X$ be a tropical variety and $x \in X$
and denote by $\sigma$ the cone of $\Sigma$ such that $x \in N(\sigma)$.
Then the tropical mutitangent space at $x$ is defined to be
\begin{align*}
\Fbb^R_p(\tau) 
= \left( \sum_{\sigma \in X \cap N(\sigma), x \in \sigma} \Lambda^p \Linear(\sigma)  \right) \cap \Lambda^p R^n
\subset \Lambda^p N(\sigma).
\end{align*}
If $\nu$ is a face of $\tau$, then there are transition maps 
$\iota_{\tau, \nu} \colon \Fbb^R_p(\tau) \to \Fbb^R_p(\nu)$
that are just inclusions if $\tau$ and $\nu$ live in the same stratum and compositions 
with projections to strata otherwise.  
\end{defn}

We denote by $\Delta_q$ the standard $q$-simplex. 
 
\begin{defn}
A smooth stratified $q$-simplex is a map $\delta \colon \Delta_q \to X$ such that:
\begin{itemize}
\item
If $\sigma$ is a face of $\Delta_q$ then there exists a polyhedron $\tau$ in $X$
such that $\mathring \sigma$ is mapped into $\mathring \tau_i$. 
\item
Let $\Delta_q = [0,\dots,q]$. If $\delta(i)$ is contained in the closure of a stratum of $N_\Sigma$, 
then so is $\delta(j)$ for $j \leq i$. 
\item
for each stratum $X_i$ of $X$ the map $\delta \colon \delta^{-1}(X_i) \to X_i$ is $C^\infty$.
\end{itemize}
We denote the free abelian group of smooth stratified $q$-simplices $\delta$
satisfying $\delta (\mathring \Delta_q) \subset \mathring \tau$ by $C_q(\tau)$. 

There is a boundary operator $\partial_{p,q} \colon C_{p,q}(X, R) \to C_{p,q-1}(X, R)$ 
that is given by the usual boundary operator on the simplex side and 
by the maps $\iota_{\tau, \nu}$ on the coefficient side, when necessary. 
Dually we have $\partial^{p,q} \colon C_{p,q}(X, R) \to C_{p,q+1}(X, R)$.
\end{defn}

\begin{defn}
The groups of smooth tropical $(p,q)$-cell and cocells are respectively 
\begin{align*}
C_{p,q}(X, R) &:= \bigoplus_{\tau \subset X} \Fbb^R_p(\tau) \otimes C_q(\tau) \\
C^{p,q}(X, R) &:= \Hom_R (C_{p,q}(X, R), R) 
\end{align*}
\end{defn}

We denote by $\Cund^{p,q}(R)$ the sheafification of $C^{p,q}(X, R)$ as defined in \cite[Definition 3.13]{JSS}
and by $\FS^p_R := \ker (\del \colon \Cund^{p,0}(R) \to \Cund^{p,1}(R))$.
Note that we have $\FS^p_R = \FS^p_\Z \otimes R$. 

\begin{defn}
We denote by 
$\HH^{p,q}_{\trop}(X) := \HH^{q}(C^{p,\bullet}(X, R), \partial) = \HH^{q}(\Cund^{p,\bullet}(X, R), \partial)$
and call this the \emph{tropical cohomology of X with coefficients in $R$}. 
Similarly we define \emph{tropical cohomology of X with coefficients in $R$ with compact support.}
\end{defn}

It was shown in \cite{JSS} that the morphisms of complexes
\begin{align*}
\FS^p_\R \to \AS^{p,\bullet} \text{ and } \FS^p_\R \to \Cund^{p, \bullet},
\end{align*}
that are given by inclusion in degree $0$ are in fact quasi isomorphisms. 
We now want to construct a de Rham morphism, 
meaning a quasi isomorphism $\dR \colon \AS^{p,\bullet} \to \Cund^{p,\bullet}$ 
that is compatible with the respective inclusions of $\FS^p_\R$.

\begin{bem} \label{rem deRham}
Let $v \otimes \delta$ be a smooth tropical $(p,q)$-cell on an open subset $\Omega$ of $X$. 
Then we define for a $(p,q)$ form $\alpha \in \AS^{p,q}(\Omega)$ 
\begin{align*}
\int_{v \otimes \delta} \alpha = \int_{\Delta_q} \delta^{-1} \langle \alpha; v \rangle 
\end{align*}
We have to argue that this integral is well defined, 
since $\delta$ might map parts of $\Delta_q$ to infinity:
Let $X_0$ be the stratum of $X$ to which the barycenter of $\Delta_q$ is mapped. 
Let $\Delta_{q,0} := \delta^{-1}(X_0)$. 
Then we have $\supp (\delta^* \alpha ) \subset \Delta_{q,0}$ by the condition of compatibility 
at the boundary for $\alpha$, hence the integral is finite. 

This defines a morphism 
\begin{align*}
\dR \colon \AS^{p,q}(\Omega) &\to C^{p,q}(\Omega) \\
\alpha &\mapsto \left( v \otimes \delta \mapsto \int_{v \otimes \delta} \alpha \right) 
\end{align*} 
and one directly verifies using the classical Stokes' theorem that this indeed induces a morphism
of complexes 
\begin{align*}
\dR \colon \AS^{p,\bullet} \to \underline{C}^{p,\bullet}.
\end{align*}
that respects the respective inclusions of $\FS^p_{\R}$. 
Since both $\Cund^{p,\bullet}$ and $\AS^{p,\bullet}$ form acyclic resolutions of $\FS^p_\R$, 
the map $\dR$ is a quasi-isomorphism. 
\end{bem}

\begin{defn}
The \emph{monodromy operator} is the unique $\AS^{0,0}$-linear map such that 
\begin{align*}
M \colon \AS^{p,q} &\to \AS^{p-1,q+1}; \\
d' x_I \wedge d''x_J &\mapsto \sum_{k = 1}^p (-1)^{p-k} d'x_{I \setminus i_k} \wedge d''x_{i_k} \wedge d''x_J.
\end{align*}
The wave operator $W \colon C^{p,q}(\R) \to C^{p-1,q+1}(\R)$ is the sheafified version of the map dual to
\begin{align*}
C_{p-1,q+1}(X, \R) &\to C_{p,q}(X, \R); \\
v \otimes \delta &\mapsto (v \wedge (\iota(\delta(1)) - \delta(0))) \otimes \delta|_{[1,\dots,q+1]}.
\end{align*}
\end{defn}

\begin{lem} \label{thm monodromy wave trop}
The diagram
\begin{align*}
\begin{xy}
\xymatrix{
&&\AS^{p-1, 1} \ar[dd]^{\dR}   \\
\FS^p_\R \ar[urr]^{(-1)^{p-1} M} \ar[drr]^{W}   \\
&&\underline{C}^{p-1,1}(\R)
}
\end{xy}
\end{align*}
commutes. 
\end{lem}
\begin{proof}
By the definition of the de Rham map we have to show that
\begin{align*}
\int_{[0,1]} \delta^* \langle M(\alpha), v \rangle = (-1)^{p-1} \langle \alpha,  W(e \otimes v) \rangle
\end{align*}
where $\delta \colon [0,1] \to X$ is a smooth stratified $1$-simplex and 
$v \in \Fbb^{p-1}(\tau)$, where $\delta((0,1)) \subset \mathring \tau$ and 
$\alpha \in \FS^{p}$.  
After picking bases and using multilinearity for both $v$ and $\alpha$ we may assume that
$\alpha = d'x_1 \wedge \dots \wedge d'x_p$ and $v = x_1 \wedge \dots \wedge x_{p-1}$. 
Then we have
\begin{align*}
\int_{[0,1]} \delta^* \langle M(\alpha), v \rangle &= 
\sum_{i=1}^p (-1)^{p-i} \int _{[0,1]} \delta^* 
\langle d'x_1 \wedge \dots \widehat{d'x_i} \wedge \dots \wedge d'x_p,  
x_1 \wedge \dots \wedge x_{p-1} \rangle \wedge d''x_i \\ 
&= 
\int_{[0,1]} \delta^* \langle d'x_1 \wedge \dots \wedge d''x_{p-1} , 
x_1 \wedge \dots \wedge x_{p-1} \rangle \wedge d''x_p \\
&= \int_{[0,1]} \delta^* d''x_{p} = dx_p(\delta(1)) - dx_{p}(\delta(0))
\end{align*}
and
\begin{align*}
\langle d'x_1 \wedge \dots \wedge d'x_{p} , \delta(1) - \delta(0) \wedge x_1 \wedge \dots \wedge x_{p-1} \rangle = 
(-1)^{p-1} dx_p({\delta(1)}) - dx_p({\delta(0)}),
\end{align*}
where $dx_p$ denotes the $p$-th coordinate functions with respect to $x_1,\dots,x_n$. 
This calculation holds true as long as the $p$-th coordinate functions is bounded on $\delta([0,1])$.
If this is not the case however, then $d''x_p$ vanishes in a neighborhood of $\delta(0)$  (resp. $\delta(1)$) 
by the compatibility condition
and we may replace $[0,1]$ by $[\varepsilon, 1]$ resp.~$[\varepsilon, 1-\varepsilon]$ resp.~$[0,1-\varepsilon]$. 
\end{proof}

\begin{satz} \label{Wave Monodromy tropical theorem}
The wave and the monodromy operator agree on cohomology up to sign by virtue of the isomorphism $\dR$, 
meaning that the diagram
\begin{align*}
\begin{xy}
\xymatrix{
\HH^{p,q}(X) \ar[rr]^{(-1)^{p-1} M} \ar[d]_{\dR} && \HH^{p-1,q+1}(X) \ar[d]^{\dR} \\
\HH^{p,q}_{\trop}(X, \R) \ar[rr]^W && \HH^{p-1,q+1}_{\trop}(X, \R)
}
\end{xy}
\end{align*}
commutes.
The same is true for cohomology with compact support. 
\end{satz}
\begin{proof}
The wave and monodromy operators give morphisms of complexes  is a morphism of complexes
\begin{align*}
W \colon \underline{C}^{p,\bullet}(\R) \to \underline{C}^{p-1,\bullet}(\R)[1] \text{ and } 
M \colon \AS^{p,\bullet} \to \AS^{p-1, \bullet}[1],
\end{align*}
hence it is sufficient to show that 
\begin{align*}
\begin{xy}
\xymatrix{
\AS^{p, \bullet} \ar[d]^{\dR} \ar[rr]^{(-1)^{p-1}M} &&\AS^{p-1, \bullet}[1] \ar[d]^{\dR}   \\
\underline{C}^{p,\bullet}(\R) \ar[rr]^W && \underline{C}^{p-1,\bullet}(\R)[1]  
}
\end{xy}
\end{align*}
commutes in the derived category. 
Replacing both $\AS^{p, \bullet}$ and $\Cund^{p, \bullet}$ 
with the quasi-isomorphic $\FS^p_\R$, 
we have to show that
\begin{align*}
\begin{xy}
\xymatrix{
&&\AS^{p-1, 1} \ar[dd]^{\dR}   \\
\FS^p_\R \ar[urr]^{(-1)^{p-1} M} \ar[drr]^{W}   \\
&&\underline{C}^{p-1,1}(\R)
}
\end{xy}
\end{align*}
commutes. 
This follows directly from Theorem \ref{thm monodromy wave trop}.
\end{proof}

\begin{prop} \label{prop wave restricts}
The wave operator descends to an operator on cohomology
\begin{align*}
W &\colon \HH^{p,q}_{\trop}(X, R) \to \HH^{p-1,q+1}_{\trop}(X, R[\Gamma]) \text{ and } \\
W &\colon \HH^{p,q}_{\trop,c}(X, R) \to \HH^{p-1,q+1}_{\trop,c}(X, R[\Gamma]),
\end{align*}
where $R[\Gamma]$ is the smallest subring of $\R$ that contains both $R$ and $\Gamma$. 
\end{prop}
\begin{proof}
We pick a triangulation of $X$ with smooth stratified simplices 
such that all vertices are $\Gamma$-points.
We can now compute $\HH^{p,q}_c(X, R)$ and
the wave homomorphism using this triangulation by \cite[Section 2.2]{MikZhar}. 
For a $(p,q)$-chain $\delta$ with respect to this triangulation and with coefficients in $R$, 
we now have that $W(\delta)$ has coefficients in $R[\Gamma]$. 
Hence $W$ restricts to a map 
$\HH^{p,q}(X, R) \to \HH^{p-1,q+1}(X, R[\Gamma])$ 
resp.~$\HH^{p,q}_c(X, R) \to \HH^{p-1,q+1}_c(X, R[\Gamma])$.  
\end{proof}

\begin{defn} \label{defn fundamental class}
We denote by 
\begin{align*}
\cap [X]_R \colon C^{n,n}_c(X, R) \to R
\end{align*}
the \emph{evaluation against the fundamental class}, as defined in \cite[Definition 4.8]{JRS}
and also the induced map on cohomology $\HH^{n,n}_{\trop,c}(X, R) \to R$. 
\end{defn}

\begin{prop} \label{prop int cap}
The following diagram commutes
\begin{align*}
\begin{xy}
\xymatrix
{
\AS^{n.n}_c(X) \ar[rr]^{\int_X} \ar[d]_{\dR} && \R \\
C^{n,n}_c(X, \R)  \ar[rru]_{\cap [X]}.
}
\end{xy}
\end{align*}
\end{prop}
\begin{proof}
This is a straightforward calculation using the definitions. 
\end{proof}

\section{Families of tropicalizations} \label{families of trop}

In this section, $K$ is a complete non-archimedean field and 
$X$ is a $K$-variety.

\subsection{Definitions}

The philosophy throughout the paper will be that we can approximate 
non-archimedean analytic spaces through embedding them into 
toric varieties and tropicalizing.
We will define families that approximate the analytic space well enough 
(fine enough families) as well as notions that tell us that 
two families basically contain the same amount of information 
(final and cofinal families).

\begin{defn} \label{defn family}
A \emph{family of tropicalizations} $\Scal$ of $X$ consists of the following data:
\begin{enumerate}
\item
A class $\Scal_{\map}$ containing closed embeddings $\varphi \colon U \to Y_\Sigma$ 
for open subsets $U$ of $X$ and toric varieties $Y_\Sigma$.
\item
For an element $\varphi \colon U \rightarrow Y_\Sigma$ of $\Scal_{\map}$ a subclass 
$\Scal_\varphi$ of $\Scal_{\map}$ 
that contains maps $\varphi' \colon U' \to Y_{\Sigma'}$ for open subsets $U' \subset U$ and 
such that there exists an affine map of toric varieties $\psi_{\varphi,\varphi'}$, such that 
\begin{align*}
\begin{xy}
\xymatrix{
U' \ar[d]_{\iota} \ar[r]^{\varphi'} & Y_{\Sigma'} \ar[d]^{\psi_{\varphi, \varphi'}} \\
U \ar[r]^{\varphi}& Y_\Sigma 
}
\end{xy}
\end{align*}
commutes.
Such a $\varphi'$ is called \emph{refinement} of $\varphi$. 
The map $\psi_{\varphi,\varphi'}$ induces an affine map of toric varieties $\Trop(\psi_{\varphi, \varphi'})$. 
The restriction of $\Trop(\psi_{\varphi, \varphi'})$ to $\Trop_{\varphi'}(U')$
 depends only on $\varphi$ and $\varphi'$, so we denote
this map by $\Trop(\varphi, \varphi') \colon \Trop_{\varphi'}(U') \to \Trop_\varphi(U)$. 

We further require that if $\varphi'$ is a refinement of $\varphi$ and $\varphi''$ is a refinement of 
$\varphi'$, then $\varphi''$ be a refinement of $\varphi$. 
\end{enumerate}
A \emph{subfamily} of tropicalizations of $\Scal$ is a family of tropicalizations $\Scal'$ 
such that all embeddings and refinements in $\Scal$ are also in $\Scal'$. 
Further $\Scal'$ is said to be a \emph{full subfamily} if whenever $\varphi, \varphi' \in \Scal'_{\map}$ 
and $\varphi'$ is a refinement of $\varphi$ in $\Scal$, it is also a refinement 
of $\varphi$ in $\Scal'$. 
\end{defn}

Foster, Gross and Payne study in \cite{FGP} so-called ``Systems of toric embeddings''. 
They only consider the case where all of $X$ is embedded into the toric variety, 
a case we later call \emph{global families} of tropicalizations. 

\begin{defn}
Let $\Scal$ be a family of tropicalizations on $X$.
An \emph{$\Scal$-tropical chart} is given by a pair $(V, \varphi)$ 
where $\varphi \colon U \to Y_\Sigma \in \Scal_{\map}$
and $V = \trop_{\varphi}^{-1}(\Omega)$ is an open subset of $\Xan$ which is the 
preimage of an open subset $\Omega$ of $\Trop_{\varphi}(U)$.

Another $\Scal$ tropical chart $(V', \varphi')$ is called an \emph{$\Scal$-tropical subchart} of $(V, \varphi)$ 
if $\varphi'$ is a refinement of $\varphi$ and $V' \subset V$. 
\end{defn}

Note that we have $\Omega = \trop_{\varphi}(V)$.

\begin{Ex}
The family of all tropicalizations is a family of tropicalizations in the 
sense of Definition \ref{defn family}. 
We denote it by $\tor$. 
\end{Ex}

In the following definition, we define the terms \emph{final} and \emph{cofinal} for
two families of tropicalizations $\Scal$ and $\Scal'$. 
While the definitions are a bit on the technical side, the idea is that if $\Scal'$
is either final or cofinal for $\Scal$, then $\Scal$-tropical charts 
provide the same information as $\Scal'$-tropical charts. 

\begin{defn} \label{defn final}
Let $\Scal$ and $\Scal'$ be families of tropicalizations.
We say $\Scal'$ is \emph{cofinal}  for $\Scal$ if for every embedding $\varphi \colon U \to Y_\Sigma$
in $\Scal_{\map}$ and every $x \in U^{\an}$ there exists
$\varphi' \colon U' \to Y_{\Sigma'}$ in $\Scal'_{\map}$ with $x \in U'^{\an}$ such that 
$\varphi'|_{U' \cap U}$ restricts to a closed embedding of $U \cap U'$ into an open torus invariant subvariety 
of $Y_{\Sigma'}$, and that embedding is a refinement of $\varphi$ in $\tor$. 

$\Scal'$ is said to be \emph{final} for $\Scal$ 
for every $\varphi \colon U \to Y_{\Sigma}$ in $\Scal_{\map}$ and $x \in U^{\an}$, 
there exists a refinement $\varphi' \colon U' \to Y_{\Sigma'}$ with $x \in U'^{\an}$, 
a closed embedding $m \colon Y_{\Sigma'} \to Y_{\Sigma''}$, 
that is an affine map of toric varieties, 
such that $m\circ \varphi'$ is in $\Scal'_{\map}$ and $\varphi'$ is a refinement of $m \circ \varphi'$ via $m$ 
in $\Scal$.  
\end{defn}

\begin{defn}
A family of tropicalizations is called \emph{fine enough} if the sets $V$ such 
that there exist $\Scal$-tropical charts $(V, \varphi)$ form 
a basis of the topology of $\Xan$ and for each pair of $\Scal$-tropical charts 
$(V_1, \varphi_1)$ and $(V_2,\varphi_2)$ there exists $\Scal$-tropical charts $(V_i, \varphi_i)_{i \in I}$,
which are $\Scal$-tropical subcharts of both $(V_1, \varphi_1)$ and 
$(V_2, \varphi_2)$, such that $V_1\cap V_2 = \cup V_i$. 
\end{defn}

\begin{lem} \label{lem fine enough}
Let $\Scal'$ be a full subfamily of $\Scal$ and assume that $\Scal$ is fine enough. 
If $\Scal'$ is final or cofinal in $\Scal$, then $\Scal'$ is also fine enough.
\end{lem}

\begin{proof}
Let $x \in \Xan$. 
Since $\Scal$ is fine enough and $\Scal'$ is a full subfamily, 
it is sufficient to prove that given an $\Scal$-tropical chart $(V, \varphi)$ with $x \in V$, 
there exists an $\Scal$-tropical chart $(V', \varphi')$ with $x \in V'$
that is a $\Scal$-tropical subchart of $(V, \varphi)$.
  
If $\Scal'$ is cofinal in $S$, then we pick $\varphi'$ as in Definition \ref{defn final}. 
Then $( V \cap U'^{\an}, \varphi')$ is a tropical chart.

If $\Scal'$ is final, then $(V \cap U'^{\an}, m \circ \varphi' )$ is a tropical chart. 
\end{proof}

\subsection{Examples}

In the following we will give examples of families of tropicalizations for a variety $X$. 
We will always specify the class $S_{\map}$ and for $\varphi \in S_{\map}$ simply define $\Scal_{\varphi}$ 
to be those $\varphi'$ where an affine map of toric varieties $\psi_{\varphi, \varphi'}$ as required 
in Definition \ref{defn family} $iii)$ exists. 
The exception to this rule are Example \ref{example linear}, where we require the map $\psi_{\varphi, \varphi'}$ 
to be a coordinate projection in order for $\varphi'$ to be a refinement of $\varphi$ and 
Example \ref{defn Gcan}. 

\begin{Ex}
The family $\A$ is the family whose 
embeddings are closed embeddings of affine open subsets of $X$ into affine space. 
This family is fine enough by the definition of the topology of $\Xan$.

Suppose we are given an embedding $\varphi \colon U \to Y_{\Sigma}$ of an open subset of $X$ 
into a toric variety with $x \in U^{\an}$. 
Let $Y_\sigma$ be an open affine toric subvariety of $Y_{\Sigma}$ such that $Y_{\sigma'}^{\an}$
contains $\varphi^{\an}(x)$.
Let $\varphi' := \varphi|_{\varphi^{-1}(Y_\sigma)}$.
Now we pick a toric embedding of $m \colon Y_\sigma \to \A^n$ as in Example \ref{trop of affine embedding}.
This shows that $\A$ is final in $\tor$.  
\end{Ex}

\begin{Ex} \label{defn G}
The family $\G$ is the family whose embeddings 
are closed embeddings of very affine open subsets of $X$ into $\G_m^n$. 

This family is also fine enough if the base field is non-trivially valued \cite[Proposition 4.16]{Gubler}, 
but not when $K$ is trivially valued \cite[Example 3.3.1]{JellThesis}.
\end{Ex}

\begin{Ex} \label{defn Gcan}
Assume that $K$ is algebraically closed. 
Let $X$ be a variety and $U$ a very affine open subset. 
Then $M = \Ocal^*(U) / K^*$ is a free abelian group of finite rank and the canonical map
$K[M] \to \Ocal(U)$ induces a closed embedding 
$\varphi_U \colon U \inj T$ for a torus $T$ with character lattice $M$. 
The embedding $\varphi_U$ is called the \emph{canonical moment map of $U$}. 
We denote by $\G_{\can}$ the family of tropicalizations 
where $\G_{\can, \map} = \{ \varphi_U \mid U \subset X \text{ very affine} \}$
and refinements being the maps induced by inclusions.  

It is easy to see that this family is cofinal in $\G$. 
\end{Ex}

\begin{defn}
A family of tropicalizations $\Scal$ for a variety $X$ is called \emph{global} 
if all $\varphi \in \Scal_{\map}$ are defined on all of $X$. 
\end{defn}

Global families of tropicalizations will play a special role, as they will allow 
us to construct classes in tropical Dolbeault cohomology. 

\begin{defn} \label{condition dagger}
We say that $X$ satisfies condition $(\dagger)$ 
if $X$ is normal and every two points in $X$ have a common affine neighborhood. 
\end{defn}

By W{\l}odarczyk's Embedding Theorem, Condition $(\dagger)$ 
is equivalent to $X$ being normal and admitting a closed embedding into a toric variety (cf.~\cite{Wlod}). 
Observe also that it is satisfied by any quasi-projective normal variety. 
It is however weaker then being normal and quasi-projective, 
as there exist proper toric varieties which are not projective. 

\begin{Ex} \label{ex systems of toric embeddings}
Let $\Scal$ be a global family of tropicalizations such that
\begin{align} \label{Payne's theorem}
\Xan = \varprojlim_{\varphi \in \Scal_{\map}} \Trop_{\varphi}(X)
\end{align}
and such that if $\varphi_1 \colon X \to Y_{\Sigma_1}$ and 
$\varphi_2 \colon X \to Y_{\Sigma_2}$ in $\Scal_{\map}$
then also $\varphi_1 \times \varphi_2 \colon X \to Y_{\Sigma_1} \times Y_{\Sigma_2} \in \Scal_{\map}$. 
Then $\Scal$ is fine enough. 
In fact it follows directly from (\ref{Payne's theorem})
that $\Scal$-tropical charts form a basis of the topology 
and from the product property that we can always locally find common subcharts. 
\end{Ex}

Families with the properties from Example \ref{ex systems of toric embeddings} 
were extensively studied in \cite{FGP}. 

When $X$ satisfies condition $(\dagger)$, this helps us say even more:

\begin{Ex}
Assume that $X$ satisfies condition $(\dagger)$. 
Then we may consider the family of tropicalizations $\tor_{\glob}$, 
where $\TS_{\glob, \map}$ is the class of
all embeddings of $X$ into toric varieties.

Let $\varphi \colon U \to \A^n$ be a closed embedding of an open subset $U$ of $X$ into an affine space
given by regular functions $f_1,\dots,f_n$ on $U$. 
Then by \cite[Theorem 4.2]{FGP} there exists an embedding $\overline{\varphi} \colon X \to Y_\Sigma$ 
such that $U$ is the preimage of an open affine invariant subvariety $U_{\sigma}$ 
and each $f_i$ is the pullback of a character that is regular on $U_\sigma$. 
This exactly shows that $\overline{\varphi}|_{U}$ is a refinement of $\varphi$ in $\tor$, 
which shows that $\tor_{\glob}$ is cofinal in $\A$. 
\end{Ex}

\begin{Ex}
Let $X$ be an affine variety and $\A_{\glob}$ the family of tropicalizations
whose class of embeddings are embeddings of all of $X$ into affine spaces. 

We show that this family is cofinal in $\A$: 
Let $U$ be an open subset of $X$ and $\varphi \colon U \to \A^n$ an closed embedding 
given by $f_1/g_1,\dots,f_n/g_n$, where $f_i, g_i$ are regular functions on $X$. 
In particular, $U = D(g_1,\dots,g_n)$. 
We pick regular functions $h_1,\dots,h_k$ on $X$ such 
that the $f_i, g_i, h_i$ generate $\Ocal(X)$. 
We consider the embedding $\varphi' \colon X \to \A^{2n} \times \A^{k}$ given by 
$f_1,\dots,f_n,g_1,\dots,g_n, h_1,\dots,h_k$.
Then $\varphi'|_{U}$ gives a closed embedding of $U$ into $\A^n \times \G_m^n \times \A^k$, 
which is clearly a refinement of $\varphi$ in $\tor$. 
\end{Ex}

\begin{Ex} \label{example linear}
Assume that $K$ is algebraically closed.
Wanner and the author considered in \cite{JellWanner} the class of \emph{linear tropical charts} for $X = \A^1$. 
In the language of the present paper, they use the following family of tropicalizations, 
which we denote by $\A_{\lin}$. 
$\A_{\lin, \map}$ is the set of linear embeddings  
i.e.~those embeddings $\varphi \colon \A^1 \to \A^r$, 
where the corresponding algebra homomorphism $K[T_1,\dots,T_r] \to K[X]$ 
is given by mapping $T_i$ to $(X-a_i)$ for $a_1,\dots,a_r \in K$. 
A refinement of $\varphi$ is then a map $\varphi' \colon \A^1 \to \A^s$ given by $(X-b_1,\dots,X-b_s)$ 
where $s > r$ and $\{a_i\} \subset \{b_j\}$. 
This family of tropicalizations is fine enough and global, for details see \cite[Section 3.2]{JellWanner}. 
By factoring the defining polynomials of any map $\A^1 \to \A^n$ into linear factors, 
it follows that $\A_{\lin}$ is cofinal in $\A_{\glob}$. 
\end{Ex}

\begin{Ex} \label{ex smooth}
Let $K$ be algebraically closed and $X$ be a smooth projective Mumford curve. 
Let $\tor_{\Smooth}$ be the class of embeddings of $X$ into toric varieties such that $\Trop_{\varphi}(X)$
is a smooth tropical curve. 
Then $\tor_{\Smooth}$ is cofinal in $\tor_{\glob}$ by \cite[Theorem A]{JellSmooth}. 
\end{Ex}

\section{Differential forms on Berkovich spaces} \label{sect superforms}

In this section, $K$ is a complete non-archimedean field
and $X$ is a variety over $K$. 
Further $\Scal$ is a fine enough family of tropicalizations.

\subsection{Sheaves of differential forms}

In this section, we define a sheaf of differential forms $\AS^{p,q}_{\Scal}$ with respect to $\Scal$. 
We will also show that for final and cofinal families, 
these sheaves are isomorphic. 
We use the sheaves $\AS^{p,q}$  of differential forms on tropical varieties 
which are recalled in Section \ref{sect tropical cohomology}. 

\begin{defn}
Let $\Scal$ be a fine enough family of tropicalizations. 
For $V$ an open subset of $\Xan$.
An element $\alpha \in \AS^{p,q}_{\Scal}(V)$ is given by a family of triples $(V_i, \varphi_i, \alpha_i)_{i \in I}$, 
where
\begin{enumerate}
\item
The $V_i$ cover $V$, i.e. $V = \bigcup_{i \in I} V_i$. 
\item
For each $i \in I$ the pair $(V_i, \varphi_i)$ is an $\Scal$-tropical chart.
\item
For each $i \in I$ we have $\alpha_i \in \AS^{p,q}(\trop_{\varphi}(V))$. 
\item
For all $i, j \in I$ there exist $\Scal$-tropical subcharts 
$(V_{ijl}, \varphi_{ijl})_{l \in L}$ that cover $V_i \cap V_j$ such that
\begin{align*}
\Trop(\varphi_i, \varphi_{ijl})^* \alpha_i = 
\Trop(\varphi_j, \varphi_{ijl})^* \alpha_j \in \AS^{p,q}(\trop_{\varphi_{ijl}}(V_{ijl})).
\end{align*} 
\end{enumerate}
Another such family $(V_j, \varphi_j, \alpha_j)_{j \in J}$ defines the same form $\alpha$ 
if their union $(V_i, \varphi_i, \alpha_i)_{i \in I \cup J}$
still satisfies iv). 

For an open subset $W$ of $V$ we can cover $W$ by $\Scal$-tropical subcharts 
$(V_{ij}, \varphi_{ij})$ of the $(V_i, \varphi_i)$. 
Then we define $\alpha\vert_W \in \AS^{p,q}_{\Scal}(W)$ to be defined by 
$(V_{ij}, \varphi_{ij}, \Trop(\varphi_{i}, \varphi_{ij})^*(\alpha_i))$. 
\end{defn}

The differentials $d'$ and $d''$ are well defined on $\AS^{p,q}_{\Scal}$ and thus we obtain
a complex $(\AS^{\bullet, \bullet}_{\Scal}, d', d'')$ of differential forms on $\Xan$. 

\begin{lem} \label{lem check one chart} \label{Lem check zero locally 2}
Let $\Scal$ be a fine enough family of tropicalization.
Let $\alpha \in \AS^{p,q}_{\Scal}(V)$ be given by a single $\Scal$-tropical chart $(V, \varphi, \alpha')$. 
Then $\alpha = 0$ if and only if $\alpha' = 0$. 
\end{lem}
\begin{proof}
This works the same as the proof in \cite[Lemme 3.2.2]{CLD}. 
\end{proof}

\begin{lem} \label{lem admissible}
Let $\Scal$ be a fine enough family of tropicalization. 
Let $\Scal'$ be a fine enough subfamily. 
Then there exists a unique morphism of sheaves
\begin{align*}
\Psi_{\Scal', \Scal} \colon \AS_{\Scal'} \to \AS_{\Scal},
\end{align*}
such that the image of a form given by a triple $(V, \varphi, \beta)$ 
is given by the same triple.
This morphism is injective. 
Furthermore if $\Scal'$ is final or cofinal, then this morphism is an isomorphism. 
\end{lem}

Recall that if $\Scal'$ is a full subfamily that is either final or cofinal in
the fine enough family $\Scal$, then $\Scal'$ is itself fine enough
by Lemma \ref{lem fine enough}.
Thus Lemma \ref{lem admissible} implies that $\Psi_{\Scal', \Scal}$
is an isomorphism. 

\begin{proof}
Injectivity follows from Lemma \ref{lem check one chart}.  

Assume that $\Scal'$ is cofinal in $\Scal$ and that we are given a form $\alpha$ given by $(V, \varphi, \beta)$. 
Fixing $x\in V$ and picking $\varphi'$ as in Definition \ref{defn final}, we define a form $\alpha'$ to be given by 
$(V \cap U'^{\an}, \varphi', \Trop(\psi_{\varphi', \varphi})^* \beta)$. 
Since $\varphi'$ was a refinement of $\varphi$ in $\tor$, 
we have 
\begin{align*}
\Psi_{\Scal, \tor}(\alpha) = \Psi_{\Scal', \tor}(\alpha') = \Psi_{\Scal, \tor}( \Psi_{\Scal', \Scal} (\alpha') )
\end{align*}
which proves $\Psi(\Scal', \Scal) (\alpha')  = \alpha$ locally at $x$ by injectivity of $\Psi_{\Scal, \tor}$.

Assume that $\Scal'$ is final in $\Scal$ and we are given a form $\alpha$ given by a tuple $(V, \varphi, \beta)$.
We pick $m$ as in Definition \ref{defn final} and define $\alpha'$ to be given by $(\varphi \circ m_2, V, \beta)$. 
Note here that $m$ induces an isomorphism of the tropicalizations of $V$, thus ``pushing $\beta$ forward 
along $\varphi$'' is possible. 
Then $\Psi_{\Scal', \Scal}(\alpha') = \alpha$ since $\varphi$ is a refinement of $\varphi \circ m$ via $m$. 
\end{proof}

%

\begin{defn}
Let $\Scal$ be a fine enough family of tropicalizations. 
We say that $\Scal$ is \emph{admissible} if $\Psi_{S, \tor}$ is an isomorphism.
\end{defn}

\begin{kor}
We have an isomorphism
\begin{align*}
\AS_{\A} \cong \AS_{\tor}.
\end{align*}
If $X$ is affine, these are also isomorphic to $\AS_{\A_{\glob}}$ and 
if $X$ satisfies condition $(\dagger)$, then these are also isomorphic to 
$\AS_{\tor_{\glob}}$.
\end{kor}

\begin{satz} \label{thm one chart}
Let $\Scal$ be a fine enough global family of tropicalizations. 
Let $\alpha \in \AS^{p,q}(V)$ be given by a \emph{finite} family 
$(V_i, \varphi_i, \alpha_i)$, where $\varphi_i \colon X \to Y_{\Sigma_i}$ in $S_{\map}$.  
Then $\alpha$ can be defined by one $\Scal$-tropical chart. 
\end{satz}
\begin{proof}
Since $\Scal$ is fine enough and global, there exists a common refinement
$\varphi \colon X \to Y_{\Sigma}$ for all the $\varphi_i$. 
Then $(V_i, \varphi)$ is an $\Scal$-tropical subchart of $(V_i, \varphi_i)$ for all $i$.  
Denote by $\alpha_i' := \Trop(\varphi_i,\varphi)^* \alpha_i$ and $\Omega_i := \trop_{\varphi}(V_i)$. 
Then $\alpha|_{V_i}$ is given by both $(V_i \cap V_j, \varphi, \alpha'_i|_{\Omega_i \cap \Omega_j})$ and 
$(V'_i \cap V'_j, \varphi, \alpha'_j|_{\Omega_i \cap \Omega_j})$. 
By Lemma \ref{Lem check zero locally 2}, 
the forms $\alpha'_i$ and $\alpha'_j$ agree on $\Omega_i \cap \Omega_j$, thus glue to give a form 
$\alpha' \in \AS^{p,q}(\trop_{\varphi}(V))$. 
The form $\alpha \in \AS^{p,q}(V)$ is then defined by $(V, \varphi, \alpha')$. 
\end{proof}

Let $\Scal$ be a global admissible family of tropicalizations. 
Let $\varphi \colon X \to Y_\Sigma$ be a closed embedding in $\Scal_{\map}$. 
We define a map $\trop^* \colon \AS^{p,q}(\Trop_{\varphi}(X)) \to \AS_{\Scal}^{p,q}(\Xan)$
by setting for $\beta \in \AS^{p,q}(\Trop_{\varphi}(X))$ the image
$\trop^*\beta$ to be the form given by the triple $(\varphi, \Xan, \beta)$.
One immediately checks that this is well defined. 
We define this similarly for forms with compact support. 

\begin{satz} \label{thm forms are limit}
Let $\Scal$ be a fine enough global family of tropicalizations. 
Let $V$ be an open subset of $\Xan$ such that there exists an $\Scal$-tropical chart $(V, \varphi)$. 
Then pullbacks along the tropicalization maps induce an isomorphism
\begin{align*}
\varinjlim \AS^{p,q}_c(\trop_{\varphi}(V)) \to \AS^{p,q}_{\Scal, c}(V)  
\end{align*}
where the limit runs over all $\Scal$-tropical charts $(V, \varphi)$. 
\end{satz}
\begin{proof}
For any $\Scal$-tropical chart $(V,\varphi)$, 
the pullback along the proper map $\trop_{\varphi}$ induces a well defined morphism 
$\AS^{p,q}_c(\trop_{\varphi}(V)) \to \AS^{p,q}_{\Scal, c}(V)$.
By definition this map is compatible with pullback between charts. 
Thus the universal property of the direct limit leads to a morphism 
$\Psi\colon\varinjlim \AS^{p,q}_{\Scal,c}(\trop_{\varphi}(V))\to \AS^{p,q}_c(V)$,
where the limit runs over all $\Scal$-tropical charts of $V$. 

This map is injective by construction. 
For surjectivity, let $\alpha \in \AS^{p,q}_c(V)$ be given by $(V_i, \varphi_i, \alpha_i)_{i \in I}$. 
Let $I'$ be a finite subset of $I$ such that the $V_i$ with $i \in I'$ cover the support of $\alpha$. 
Then $(V_i, \varphi_i, \alpha_i)_{i \in I'}$ defines $\alpha|_{V'} \in \AS^{p,q}_{\Scal,c}(V')$, 
where $V' = \bigcup_{i \in I'}V_i$. 
Since $\Scal$ is global and fine enough, the conditions of Theorem \ref{thm one chart} are satisfied and 
$\alpha\vert_{V'}$ can be defined by a triple $(V', \varphi', \alpha')$.
By passing to a common refinement with $(V, \varphi)$ 
we may assume that $(V', \varphi')$ is a subchart of $(V, \varphi)$. 

It follows from Lemma \ref{Lem check zero locally 2} that 
$\supp(\alpha') = \varphi'_{\trop}(\supp(\alpha))$ (cf.~\cite[Corollaire 3.2.3]{CLD}). 
Thus $\supp(\alpha')$ is compact.
We extend $\alpha'$ by zero to $\tilde{\alpha}'  \in \AS^{p,q}_{\Trop_{\varphi'}(X),c}(\trop_{\varphi'}(V))$.
Then $\alpha$ is defined by $(V, \varphi', \tilde{\alpha}')$, which is in the image of $\Psi$. 
\end{proof}

\subsection{Tropical Dolbeault cohomology}

In this section we assume that $\Scal$ is a admissible 
family of tropicalizations and often just write $\AS^{p,q}$ for $\AS^{p,q}_{\Scal}$. 

\begin{defn}
We define \emph{tropical Dolbeault cohomology} to be the cohomology
of the complex $(\AS^{p,\bullet}(\Xan), d'')$, i.e.
\begin{align*}
\HH^{p,q}(\Xan) := 
\frac {\ker (d'' \colon \AS^{p,q}(\Xan) \to \AS^{p,q+1}(\Xan) } 
{\im (d'' \colon \AS^{p,q-1}(\Xan) \to \AS^{p,q}(\Xan))}.
\end{align*}
Similarly we define cohomology with compact support $\HH^{p,q}_c(\Xan)$
as the cohomology of forms with compact support. 
\end{defn}

\begin{satz} \label{thm cohomology limit}
Let $\Scal$ be an admissible global family of tropicalizations.  
Then pullbacks along tropicalization maps induce an isomorphism
\begin{align*}
\varinjlim_{\varphi \in \Scal} \HH_c^{p,q}(\Trop_\varphi(X)) \to \HH^{p,q}_c(\Xan) .
\end{align*}
\end{satz}
\begin{proof}
The map is induced by $\trop^*$. 
The theorem follows from the fact that taking cohomology commutes with forming direct limits and
Theorem \ref{thm forms are limit}. 
\end{proof}

\begin{bem} \label{bem non-compact non-true}
The corresponding statement for $\HH^{p,q}(\Xan)$ fails.
We sketch the argument. 
If $X$ is an affine variety, then we may pick $\Scal = \A_{\glob}$, the class
of closed embeddings into affine space. 

One can show that $\HH^{n,n}(Y) = 0$ for all tropical subvarieties of $\T^n = \Trop(\A^n)$. 
Thus we have $\varinjlim_{\varphi \in \Scal} \HH^{p,q}(\Trop_\varphi(X)) = 0$. 

Let $K$ be algebraically closed, $E$ be an elliptic curve of good reduction and let $e$ be a rational point of $E$. 
Let $V$ be an open neighborhood of $e$ that is isomorphic to an open annulus
and let $X  = E \setminus e$. 

Using the Mayer-Vietoris sequence for the cover $(\Xan, V)$ for $E^{\an}$ 
gives the following exact sequence 
\begin{align*}
\HH^{1,0}(V \setminus e) \to \HH^{1,1}(E^{\an}) \to \HH^{1,1}(\Xan) \oplus \HH^{1,1}(V). 
\end{align*}
It was shown in \cite[Theorem 5.7]{JellWanner} that both 
$\HH^{1,0}(V \setminus e)$ and $\HH^{1,1}(V)$ are finite dimensional. 
Further it was shown in \cite[Theorem B]{JellDuality} that when the residue field of $K$ is $\C$, 
then $\HH^{1,1}(E^{\an})$ is infinite dimensional. 
This implies that $\HH^{1,1}(\Xan)$ is infinite dimensional. 
In particular, it is not equal to $\varinjlim_{\varphi \in \Scal} \HH^{p,q}(\Trop_\varphi(X)) = 0$.
\end{bem}

\section{Comparison theorems} \label{sect comparison theorems}

\subsection{Comparing with Gubler's definition} 

Since Gubler's definition only works when $K$ is non-trivially valued and algebraically closed
we assume for this subsection that this is the case.

\begin{satz} \label{thm G admissible}
The family of tropicalizations $\G$ from Example \ref{defn G} is admissible. 
\end{satz}

This theorem is not a formal consequence of the definition, 
since the family $\G$ does not see any boundary considerations. 
Since for algebraically closed $K$, 
Gubler defined in \cite{Gubler} the sheaf $\AS$ 
like we here defined $\AS_{\G_{\can}}$ this theorem
is crucial for us as it proves that what we show in later sections actually 
applies to $\AS$. 

\begin{proof}
We have to show that the map $\Psi := \Psi_{\G, \tor} \colon \AS_{\G} \to \AS_{\tor}$
is surjective. 
We argue locally around a point $x \in \Xan$. 
Let $\alpha \in \AS_{\tor}$ be locally given by $(V, \varphi, \beta)$,
where $\varphi \colon U \to Y_\Sigma$ is a closed embedding of an open subset $U$ of $X$
into a toric variety.  
We fix coordinates on the torus stratum $T$ of $Y_\Sigma$ that $x$ is mapped to under $\varphi^{\an}$, 
identifying $T$ with $\G_m^n$ for some $n$
and we denote by $Z := \varphi^{-1}(T)$. 
Since very affine open subset form a basis of the topology of $X$, there exists 
a closed embedding $\varphi' \colon U' \to \G_m^r$ for $x \in U'^{\an} \subset U^{\an}$ 
such that if $\pi \colon \G_m^r \to \G_m^n$ is 
the projection to the first $n$ factors, then $\varphi|_Z = \pi \circ \varphi' \circ \iota_Z$. 

After shrinking $V$ we may assume that $\trop_{\varphi}(V)$ is a neighborhood of $\trop_{\varphi}(x)$, 
on which $\beta = \pi^*(\beta_0)$ holds, where $\pi$ is the projection to the stratum $\Trop(T)$ of $\Trop(Y)$. 
We define the form $\alpha' \in \AS_{\G}$ to be given by $(V', \varphi', \Trop(\psi)^* \beta_0)$. 
Since $\Psi(\alpha')$ is by definition given by the same triple, 
we have to show that $(V, \varphi, \beta)$ and $(V', \varphi', \Trop(\psi)^*\beta_0)$ define the same form
in a neighborhood of $x$. We do that by pulling back to a common subcharts, namely 
$\varphi \times \varphi' \colon U' \to Y_\Sigma \times \G_m^r$.
When pulling back on the tropical side we find that both $\Trop(\pi_1) ^* \beta$ as well 
as $\Trop(\pi_2)^* \beta_0$ are simply the pullback of $\beta_0$, hence 
these forms agree indeed. 
\end{proof}

\begin{kor} \label{Cor Gm admissible}
The family of tropicalizations $\G_{\can}$ is admissible.
\end{kor}
\begin{proof}
This follows from Lemma \ref{lem admissible}, the fact that $\G_{\can}$ is cofinal in $\G$ and 
Theorem \ref{thm G admissible}. 
\end{proof}

\subsection{Comparing with the analytic definition by Chambert--Loir and Ducros}

We denote by $\AS^{p,q}_{\an}$ the analytically defined sheaf of differential forms by 
Chambert-Loir and Ducros \cite{CLD}.

\begin{defn}
We define a map
\begin{align*}
\Psi_{\an}\colon \AS^{p,q}_{\A} \to \AS^{p,q}_{\an}.
\end{align*}
Let $(V, \varphi, \alpha)$ an $\A$-tropical chart, where $\varphi$ is given by $f_1,\dots,f_r$.
Let $f_{i_1},\dots,f_{i_k}$ be those $f_i$ that do not vanish at $X$. 
Then those induce a map $\varphi' \colon V \to \G_m^{k}$
and $\trop_{\varphi'}(V) = \trop_{\varphi}(V) \cap \{ z_{i_j} = -\infty \text{ for all } j = 1,\dots,k\}$,
which is a stratum of $\trop_{\varphi}(V)$ and denote by $\alpha_I$ the restriction of 
$\alpha$ to that stratum.. 
Then we define $\Psi(\alpha)$ to be given by $(V, \varphi_{\G_m}, \alpha_I)$. 
\end{defn}

\begin{satz}
The map $\Psi_{\an}$ is an isomorphism.
\end{satz}
\begin{proof}
We only have to prove the statement on stalks, so we fix $x \in \Xan$. 
By construction any form is around $x$ given by one $\A$-tropical chart. 
Then $\Psi_{\an}(\alpha)$ is also given around $x$ by one chart. 
If now $\Psi_{\an}(\alpha) = 0$, then $\alpha_I$ equals zero by \cite[Lemme 3.2.2]{CLD} and 
thus $\alpha$ equals zero.
This shows injectivity. 

To show surjectivity, we take a form $\alpha$ that is given locally around $x$ by a map 
$V \to \G_m^{r, \an}$ defined by $f_1, \dots, f_r$
and a form $\alpha_1$ on $\Trop_{f_1,\dots,f_r}(V)$.  
The $f_r$ can be expressed as Laurent series and around $x$ we may cut them of in 
sufficiently high degree to obtain Laurent polynomials 
$f'_1 \dots f'_r$ such that $\vert f_i \vert = \vert f'_i \vert$ for all $i$ on a neighborhood $V_1$ of $x$. 
Now choosing functions $g_1, \dots,g_s \in \Ocal_X(U)$ for an open subset $U$ of $X$ 
such that $x \in U^{\an}$, the $f'_i$ and $g_i$ define a closed embedding $\varphi \colon U \to \A^{r+s}$ 
with the property that there exists a neighborhood $V_2$ of $x$ in $V_1$ and a tropical chart $(V_2, \varphi)$.
We may assume that $g_1,\dots,g_t$ are non vanishing at $x$ and $g_{t+1},\dots,g_s$ are. 
Now $\alpha|_{V_2}$ is defined by a form $\alpha_2$ on $\Trop_{f'_1,\dots,f'_r,g_1,\dots,g_t}(V_2)$. 
By construction we have 
$\Trop_{f'_1,\dots,f'_r,g_1,\dots,g_t}(V_2)$ is the stratum of $\Trop_{\varphi}(V_2)$, where the coordinates 
$t+1$-th to $s$-th coordinate are $- \infty$. 
Denote by $\pi \colon \Trop_{\varphi}(V_2) \to \Trop_{f'_1,\dots,f'_r,g_1,\dots,g_t}(V_2)$ the projection
Let $\alpha_3 = \pi^* \alpha_2$. 
Then we define $\beta \in \AS^{p,q}_{\A}(V_2)$ by $(V_2, \varphi, \alpha_3)$. 
Now $\Psi_{\an}(\beta)$ is given by $V_2 \to \G_m^{r+t}$ given by $f'_1,\dots,f'_r,g_1,\dots,g_t$ and $\alpha_2$.
Since $\trop_{f_1,\dots,f_r,g_1,\dots,g_t} = \trop_{f'_1,\dots,f'_r,g_1,\dots,g_t} \colon V_2 \to \R^{r+t}$, 
the result follows from \cite[Lemme 3.1.10]{CLD}. 
\end{proof}

\begin{lem} \label{lem tracking through maps}
Let $K$ be non-trivially valued and $\alpha \in \AS^{p,q}_{\G}(V)$ given by 
$(V, \varphi, \alpha')$. 
Then $\Psi_{\an} \circ \Psi_{\A, \tor}^{-1} \circ \Psi_{\G, \tor} (\alpha)$ is given by $(V, \varphi^{\an}, \alpha')$. 
\end{lem}
\begin{proof}
Let $\varphi \colon U \to \G_m^r$ be given by invertible functions $f_1,\dots,f_r$.
Then the corresponding embedding into $\A^{2r}$, which we denote by $\varphi_{\pm}$ 
is given by $f_1,f_1^{-1},\dots,f_r,f_r^{-1}$.
Denote by $\pi \colon \R^{2r} \to \R^r$ the projection to the odd coordinates. 
Then by construction $\Psi^{\an} \circ \Psi_{\G_m}(\alpha)$ is given by $(V, \varphi_{\pm}^{\an}, \pi^*(\alpha'))$. 
Since $\varphi_{\pm}^{\an}$ is a refinement of $\varphi^{\an}$ and the map induced on tropicalizations 
is precisely $\pi$, we find that 
$(V, \varphi_{\pm}^{\an}, \pi^*(\alpha)) = (V, \varphi^{\an}, \alpha') \in \AS^{p,q}_{\an}(V)$, 
which proves the claim. 
\end{proof}

\section{Integration} \label{sect integration}

In this section we denote by $n$ the dimension of $X$
and we let $\Scal$ be a fine enough family of tropicalizations.

\begin{defn}
Let $\alpha \in \AS^{n,n}_{\Scal,c}(\Xan)$ be an $(n,n)$-form with compact support. 
A \emph{tropical chart of integration for $\alpha$} is a $\Scal$-tropical chart $(U^{\an}, \varphi)$ 
where $U$ is an open subset of $X$ and $\varphi \colon U \to Y_{\Sigma}$ 
is a closed embedding of $U$ into a toric variety, 
such that $\alpha|_{U^{\an}}$ is given by 
$(U^{\an}, \varphi, \beta)$ for $\beta_U \in \AS^{n,n}_c(\Trop_{\varphi}(U))$. 
\end{defn}

\begin{lem}
There always exists an $\A$-tropical chart of integration. 
If $K$ is non-trivially valued and algebraically closed, there always exist $\G$-tropical 
charts of integration. 
If $\Scal$ is global, then there always exist 
$\Scal$-tropical charts of integration. 
\end{lem}
\begin{proof}
The existence of $\A$-tropical charts of integration is proved in \cite[Lemma 3.2.57]{JellThesis} and 
the existence of $\G$-tropical charts of integration is proved in \cite[Proposition 5.13]{Gubler}. 
Note that \cite{JellThesis} assumes that $K$ is algebraically closed, but the proof goes through here.  

In the case of global $\Scal$, let $\alpha$ be given by $(V_i, \varphi_i, \alpha_i)_{i \in I}$. 
Then we may pick finitely many $i$ such that $\supp(\alpha)$ is covered by the $V_i$ and apply Theorem 
\ref{thm one chart} to obtain an $\Scal$-tropical chart of integration. 
\end{proof}

\begin{satz} \label{thm independence of integral}
Let $\alpha \in \AS^{n,n}_{\Scal, c}(\Xan)$ 
and $(U^{\an}, \varphi, \alpha_U)$ an $\Scal$-tropical chart of
integration for $\alpha$. Then the value
\begin{align*}
\int_{\Xan}^{\Scal} \alpha := \int_{\Trop_{\varphi}(U)} \alpha_U
\end{align*}
depends only on $\alpha$, in the sense that it is independent of the triple $(U, \varphi, \alpha_U)$
representing $\alpha$. 
Further it is also independent of $\Scal$, in the sense that
\begin{align*}
\int^{\Scal}_{\Xan} \alpha = \int^{\tor}_{\Xan} \Psi_{\Scal, \tor}(\alpha).
\end{align*}
\end{satz}
\begin{proof}
Assume first that $K$ is non-trivially valued and algebraically closed. 
Then Gubler showed in \cite[Lemma 5.15]{Gubler} that 
as a consequence of the Sturmfels-Tevelev-formula \cite{StTe, BPR}
the first part of the statement holds for $\Scal = \G$. 

Let $\varphi \colon U \to Y_\Sigma$. 
We may assume that $\varphi(U)$ meets the dense torus $T$. 
Denote by $\mathring U = \varphi^{-1}(T)$ and 
$\mathring \varphi := \varphi|_{\mathring U}$. 
We have $\supp(\alpha) \subset \mathring U$ 
and $\supp(\alpha_U) \subset \Trop_{\varphi|_{\mathring U}}(\mathring U)$ 
by \cite[Proposition 3.2.56]{JellThesis},
which implies that $({\mathring U}^{\an}, \varphi, \alpha_U|_{\Trop_{\varphi}(\mathring U)})$
is a $\G$-tropical chart of integration for $\alpha$. 
Thus we get
\begin{align} \label{eq integral fundamental}
\int^{\Scal}_{\Xan} \alpha = \int_{\Trop_{\varphi}(U)} \alpha_U = 
\int_{\Trop_{\varphi}(\mathring U)}
\alpha_U|_{\Trop_{\varphi}(\mathring U)} = 
\int^{\G}_{\Xan} \Psi_{\Scal, \tor} \circ \Psi_{\G,\tor}^{-1} \alpha
\end{align}
The right hand side of equation (\ref{eq integral fundamental}) 
does not depend on $U, \varphi$ or $\alpha_U$ by Gubler's 
result, hence neither does the left hand side. So we proved the first part of the theorem when
$K$ is non-trivially valued. 
Then second part follows also from (\ref{eq integral fundamental}) applied once to $\Scal$
and once to $\tor$.

We reduce to the non-trivially valued case by picking a non-archimedean, complete, algebraically 
closed non-trivially valued extension $L$ of $K$.  
Let $p\colon X_L \rightarrow X$ be the canonical map. 
Since tropicalization is invariant under base field extension (cf.~\cite[Section 6 Appendix]{Payne}), we can define
$\alpha \in \AS^{n,n}_c(\Xan_L)$ to be given by $(U^{\an}_L, \varphi_L, \alpha_U)$.  
Now we have 
\begin{align*}
\int_{\Trop_{\varphi}(U)} \alpha_U = \int_{\Trop_{\varphi_L}(U_L)} \alpha_U
\end{align*}
and the right hand side depends only on $\alpha_L$, which depends only on $\alpha$.  
The last part follows because the maps $\Psi$ are also compatible with base change. 
\end{proof}

\begin{defn} 
We define 
\begin{align*}
\int_{\Xan} \alpha = \int_{\Xan}^{\tor} \alpha_U.
\end{align*}
\end{defn}

\begin{lem} \label{lem base change 1}
$\int_{\Xan}$ does not change when extending the base field.
\end{lem}
\begin{proof}
This follows from the last part of the proof of Theorem \ref{thm independence of integral}.
\end{proof}

Chambert-Loir and Ducros also define an integration for $(n,n)$-forms with compact 
support, which we denote by $\int^{\CLD}$.

\begin{lem} \label{lem base change 2}
$\int^{\CLD}$ does not change when extending the base field.
\end{lem}

\begin{proof}
The base changes of any atlas of integration in the sense of \cite{CLD} is still an atlas of integration
for the base changed form. 
Then since tropicalizations and the multiplicity $d_D$ from their definition does not change, we obtain the result. 
\end{proof}

\begin{satz}
Let $\alpha \in \AS_c^{n,n}(\Xan)$. 
Then 
\begin{align*}
\int^{\CLD}_{\Xan} \Psi_{\an}(\alpha) = \int_{\Xan} \alpha.
\end{align*}
\end{satz}
\begin{proof}
Let $L$ be a non-trivially valued non-archimedean extension of $K$. 
After replacing $K$ by $L$, $X$ by $X_L$ and $\alpha$ be $\alpha_L$, 
we may, since neither integral changes by Lemmas \ref{lem base change 1} and \ref{lem base change 2}, 
assume that $K$ is non-trivially valued. 

Then, by Corollary \ref{Cor Gm admissible}, $\G_{\can}$ is an admissible family of tropicalizations 
and thus by Theorem \ref{thm independence of integral}, we may take $\alpha \in \AS^{n,n}_{\G_{\can}, c}$, 
say given by $(V_i, \varphi_i, \alpha_i)_{i \in I}$. 
Then by Lemma \ref{lem tracking through maps} we have to show 
\begin{align*}
\int_{\Xan}^{\CLD} (V_i, \varphi_i^{\an}, \alpha_i)_{i \in I} = \int_{\Xan} (V_i, \varphi_i, \alpha_i)_{i \in I}.
\end{align*}
This was precisely shown in \cite[Section 7]{Gubler}. 
\end{proof}

\section{Tropical cohomology} \label{sect tropical cohomology Xan}

In this section we assume that $\Scal$ is a fine enough family of tropicalizations for $X$
that is cofinal in $\tor$. 
We let $R$ be a subring of $\R$. 
We will use the sheaf of tropical cochains $\Cund^{p,q}(R)$ and the constructions 
from Section \ref{sect tropical cohomology}. 

\begin{defn} \label{defn analytic cocycles}
Let $V$ be a open subset of $\Xan$. 
An element of $\Cund^{p,q}_{\Scal}(V, R)$ is given by a family 
$(V_i, \varphi_i, \eta_i)_{i \in I}$
such that:
\begin{enumerate}
\item
The $V_i$ cover $V$, i.e. $V = \bigcup_{i \in I} V_i$. 
\item
For each $i \in I$ the pair $(V_i, \varphi_i)$ is an $\Scal$-tropical chart.
\item
For each $i \in I$ we have $\eta_i \in \Cund^{p,q}(\trop_{\varphi_i}(V_i), R)$. 
\item
For all $i, j \in I$ there exist $\Scal$-tropical subcharts 
$(V_{ijl}, \varphi_{ijl})_{l \in L}$ that cover $V_i \cap V_j$ such that
\begin{align*}
\Trop(\varphi_i, \varphi_{ijl})^* \eta_i = 
\Trop(\varphi_j, \varphi_{ijl})^* \eta_j \in \Cund^{p,q}(\trop_{\varphi_{ijl}}(V_{ijl}), R).
\end{align*} 
\end{enumerate}
Another such family $(V_j, \varphi_j, \eta_j)_{j \in J}$ 
defines the same form if and only if their union $(V_i, \varphi_i, \eta_i)_{i \in I \cup J}$
still satisfies iv). 
\end{defn}

For each $p$ we obtain a complex of sheaves $(\Cund^{p,\bullet}_{\Scal}(R), \partial)$ on $\Xan$.
If $\Scal = \tor$ we will drop the subscript and write $\Cund^{p,q}(R)  := \Cund^{p,q}_{\tor}(R)$.

\begin{bem}
It follows the same way as in Section \ref{sect superforms} for $\AS^{p,q}$ that $\Cund^{p,q}_{\Scal}$ 
is isomorphic to $\Cund^{p,q}_{\Scal'}$ when $\Scal'$ is final or cofinal for $\Scal$.

The author does not know whether one gets isomorphic sheaves when one 
applies Definition \ref{defn analytic cocycles} with for example $\Scal = \G$. 
The missing piece here is that for differential forms on tropical toric varieties, 
the condition of compatibility requires every forms to locally be a pullback of a forms 
from a tropical torus (i.e.~$\R^n$). 
The same is not true for tropical cocycles, hence the arguments form Theorem \ref{thm G admissible} 
do not work. 
\end{bem}

\begin{lem} \label{lem check zero locally cochains}
Let $\eta \in \Cund^{p,q}(V, R)$ be given by $(V, \varphi, \eta')$. 
Then $\eta = 0$ if and only if $\eta' = 0$. 
\end{lem}
\begin{proof}
The proof for forms \cite[Lemma 3.2.12]{JellThesis} works word for word. 
\end{proof}

\begin{prop} \label{prop comparison coefficients}
Comparing with tropical cohomology, we obtain
\begin{align*}
\Cund^{p,q}_c(\Xan, R) = \varinjlim_{\varphi \in \Scal} \Cund^{p,q}_{\trop, c}(\Trop_{\varphi}(X), R)
\end{align*}
and
\begin{align*}
\HH^{p,q}_{\trop, c} (\Xan, R) = \varinjlim_{\varphi \in \Scal} \HH^{p,q}_{\trop, c}(\Trop_{\varphi}(X), R).
\end{align*}
\end{prop}
\begin{proof}
This follows from Lemma \ref{lem check zero locally cochains} in the same way as for forms in 
Theorems \ref{thm forms are limit} and \ref{thm cohomology limit} follow from Lemma \ref{lem check one chart}.
\end{proof}

\begin{defn}
The maps $\dR$ defined in Section \ref{sect tropical cohomology} 
define maps
\begin{align*}
\dR \colon \AS^{p,q} \to \Cund^{p,q}(\R)
\end{align*}
the induces a morphism of complexes of sheaves. 
\end{defn}

\begin{defn}
The cohomology 
\begin{align*}
\HH^{p,q}_{\trop}(\Xan, R) := \HH^q(\Cund^{p,\bullet}(\Xan, R), \partial)
\end{align*}
is called \emph{tropical cohomology with coefficients in $R$} of $\Xan$. 
Similarly 
\begin{align*}
\HH^{p,q}_{\trop, c}(\Xan, R) := \HH^q(\Cund_c^{p,\bullet}(\Xan, R), \partial)
\end{align*}
is called \emph{tropical cohomology with coefficients in $R$ with compact support} of $\Xan$.
\end{defn}

\begin{defn}
We denote by $\FS^p_R := \ker (\partial \colon \Cund^{p,0}(R) \to \Cund^{p,1}(R))$.
\end{defn}

\begin{lem} \label{lem complex exact}
The complex 
\begin{align*}
0 \to \FS^p_R \to \Cund^{p,0}(R) \to \Cund^{p,1}(R) \to \dots \to \Cund^{p,n}(R) \to 0
\end{align*}
is exact.  
\end{lem}
\begin{proof}
Exactness on the tropical side is true by \cite[Proposition 3.11 \& and Lemma 3.14]{JSS} 
(with real coefficients, but the proof goes through here). 
It is then automatically true on the analytic side using the definitions 
(cf.~the proof for forms \cite[Theorem 4.5]{Jell}). 
\end{proof}

\begin{bem}
The sheaves $\AS^{p,q}_{\Scal}$ admit partitions of unity, 
which can be shown the same way as it was shown by Gubler for 
$\Scal = \G$ in \cite[Proposition 5.10]{Gubler}.
This proof however uses the $\R$-structure of those sheaves. 

The sheaves $\Cund^{p,q}(R)$ on a tropical variety (as defined in Section \ref{sect tropical cohomology}) 
are flasque sheaves \cite[Lemma 3.14]{JSS}, hence in particular acyclic. 

However, it is not clear whether this property also holds for $\Cund^{p,q}(R)$
on the analytic space $\Xan$ in general. 
We will prove some partial results in the next Lemma.
\end{bem}

Recall condition $(\dagger)$ from Definition \ref{condition dagger}.

\begin{lem} \label{lem cocycles acyclic}
Assume that $R = \R$ or that $X$ satisfies condition $(\dagger)$.
Then the sheaves $\Cund^{p,q}(R)$ are acyclic with respect to the 
functor of global sections as well as global sections with compact support. 
\end{lem}
\begin{proof}
Using the map $\dR \colon \AS^{0,0} \to \Cund^{0,0}(\R)$,  
we see that $\Cund^{0,0}(\R)$ admits partitions of unity. 
Hence $\Cund^{0,0}(\R)$ is a fine sheaf and since $\Cund^{p,q}(\R)$
is a $\Cund^{0,0}(\R)$-module (via the cap product) 
we see that $\Cund^{p,q}(\R)$ is also a fine sheaf. 

In general, if $X$ satisfied condition ($\dagger$)
then $\tor_{\glob}$ is a global family of tropicalizations that 
is cofinal in $\A$, which is final in $\tor$. 
hence $\Cund_{\tor_{\glob}}^{p,q}(R) \cong \Cund^{p,q}(R)$.
Any section of $\Cund_{\tor_{\glob}}^{p,q}(R)$ 
that is defined by finitely many charts $(V_i, \varphi_i, \eta_i)$ 
can be defined by a single chart. 
This can be shown the same way as for forms in Theorem \ref{thm one chart}. 

Since any section over a compact subset of $\Xan$ is 
defined by finitely many charts, each such section
can be defined by one chart $(\Xan, \varphi, \eta)$. 
Now since the sheaf $\Cund^{p,q}(R)$ on $\Trop_{\varphi}(X)$
is flasque, this section can be extended to a global section. 
This shows that the sheaf $\Cund^{p,q}(R)$ on $\Xan$ is c-soft in the sense of 
\cite[Definition 2.5.5]{KS}, 
which implies that it is acyclic for the functor of global sections with compact support 
\cite[Proposition 2.58 \& Corollary 2.5.9]{KS}. 

Since $\Cund^{p,q}(R)$ is c-soft, to show that it is acyclic for the functor of global sections, 
we have to show that $\Xan$ admits a countable cover by compact sets \cite[Proposition 2.5.10]{KS}. 
Since $\A^{n, \an}$ is covered by countably many discs, this holds if $X$ is affine. 
Since general $X$ is covered by finitely many affine varieties, the claim follows. 
\end{proof}

\begin{kor} \label{kor sheaf cohomology}
If $R = \R$ or if $X$ satisfies condition $(\dagger)$ we have 
\begin{align*}
\HH^{p,q}_{\trop}(\Xan, R) = \HH^{q}(\Xan, \FS^p_{R}) \text{ and } 
\HH^{p,q}_{\trop,c}(\Xan, R) = \HH^{q}_c(\Xan, \FS^p_{R}).
\end{align*}
Further, we have
\begin{align*}
\HH^{p,q}_{\trop}(\Xan, R) =\HH^{p,q}_{\trop}(\Xan, \Z) \otimes R.  \text{ and } 
\HH^{p,q}_{\trop,c}(\Xan, R) =\HH^{p,q}_{\trop,c}(\Xan, \Z) \otimes R.
\end{align*}

\end{kor}
\begin{proof}
This follows from Lemmas \ref{lem complex exact} and \ref{lem cocycles acyclic}, 
the fact that since $R$ is torsion free, thus a flat $\Z$-module 
and $\FS^{p}_R = \FS^p_\Z \otimes \Z$. 
\end{proof}

\begin{satz} [Tropical analytic de Rham theorem] \label{Tropical analytic de Rham theorem}
There exist isomorphisms
\begin{align*}
\HH^{p,q}(\Xan) \cong \HH^{p,q}_{\trop}(\Xan, \R) \text{ and } 
\HH^{p,q}_c(\Xan) \cong \HH^{p,q}_{\trop, c}(\Xan, \R).
\end{align*}
that are induced by the de Rham morphism on the tropical level,
as defined in Remark \ref{rem deRham}.
\end{satz}
\begin{proof}
We have a map
\begin{align*}
\dR \colon \AS^{p,q} \to \Cund^{p,q}(\R)
\end{align*}
that is locally given by using the de Rham map on the tropical side constructed in Remark \ref{rem deRham}.
This makes the following diagram commutative:
\begin{align*}
\begin{xy}
\xymatrix{
&&\AS^{p, \bullet} \ar[dd]^{\dR}   \\
\FS^p \ar[urr] \ar[drr] \\
&&\underline{C}^{p,\bullet}. 
}
\end{xy}
\end{align*}
This is now a commutative diagram of acyclic resolutions of $\FS^p$, 
which proves the theorem. 
\end{proof}

\section{Wave and monodromy operators} \label{sect wave and monodromy operator}

Since both the monodromy operator $M$ on superforms and the wave operator defined 
in Section \ref{sect tropical cohomology} on
tropical cochains commute with pullbacks along affine maps on tropical toric varieties, 
we obtain maps
\begin{align*}
M &\colon \HH^{p,q}(\Xan) \to \HH^{p-1,q+1}(\Xan)  \text{ and}\\
W &\colon \HH^{p,q}_{\trop}(\Xan, \R) \to \HH^{p-1,q+1}_{\trop}(\Xan, \R).
\end{align*}

\begin{satz} \label{kor wave mondromy analytic}
The wave and the monodromy operator agree on cohomology up to sign by virtue of the isomorphism $\dR$, 
meaning that the diagram
\begin{align*}
\begin{xy}
\xymatrix{
\HH^{p,q}(\Xan) \ar[rr]^{(-1)^{p-1} M} \ar[d]_{\dR} && \HH^{p-1,q+1}(\Xan) \ar[d]^{\dR} \\
\HH^{p,q}_{\trop}(\Xan, \R) \ar[rr]^W && \HH^{p-1,q+1}_{\trop}(\Xan, \R)
}
\end{xy}
\end{align*}
commutes.
The same is true for cohomology with compact support. 
\end{satz}
\begin{proof}
The proof of Theorem \ref{Wave Monodromy tropical theorem} works word for word. 
\end{proof}

In \cite{Liu2}, Liu defined a $\Q$-subsheaf of $\FS^{p}_{\R}$ and defined \emph{rational classes} 
in tropical Dolbeault cohomology. 

\begin{defn} [Liu]
Denote by $\mathcal{J}^p$ the $\Q$-subsheaf of $\FS^p$ 
generated by sections of the form $(V, \varphi, \alpha)$, 
where $\varphi \colon U \to T$ and $\alpha \in \Lambda^p M$. 
The classes in $\HH^{p,q}(\Xan, \mathcal{J}^p) \subset \HH^{p,q}(\Xan, \R)$ 
are called \emph{rational classes}.
\end{defn}

\begin{prop} \label{prop rational classes}
Assume that $X$ satisfies condition $(\dagger)$. 
We have isomorphisms $\FS^p_\Q = \mathcal{J}^p$
and $\HH^{p,q}(\Xan, \Q) = \HH^{p,q}(\Xan, \mathcal{J}^p)$. 
\end{prop}
\begin{proof}
The explicit computation \cite[Proposition 3.11]{JSS} of $\FS^p_\R$ works also for rational coefficients. 
Then this follows directly from the definitions and Corollary \ref{kor sheaf cohomology}. 
\end{proof}

The following statement  in particular shows that 
$\HH^{p,q}(\Xan, \Z) \subsetneq \HH^{p,q}(\Xan, \Q)$. 

\begin{satz}
Assume that $X$ satisfies condition $(\dagger)$. 
Then is a non-trivial $R$-linear map
\begin{align*}
\cap [\Xan]_R \colon \HH^{n,n}_{\trop, c}(\Xan, R) \to R.
\end{align*} 
If $R = \R$, then this agrees with the map induced by integration via $\dR$. 
\end{satz}
\begin{proof}
The maps $[\Trop_{\varphi}(X)]_R$ as defined in Definition \ref{defn fundamental class}
are compatible with pullback along refinements, so by Proposition \ref{prop comparison coefficients}
we get a well defined map on $\HH^{n,n}_c(\Xan, R)$. 

The last part of the statement follows from Proposition \ref{prop int cap}. 
\end{proof}

Liu showed that if the value group of $K$ is equal to $\Q$, then 
his monodromy map $M$ maps rational classes to rational classes \cite[Theorem 5.5 (1)]{Liu2}. 
We generalize to the following statement:

\begin{satz} 
Assume that $X$ satisfies condition $(\dagger)$.  
The wave operator $W$ (and by virtue of Corollary \ref{kor wave mondromy analytic} also
the monodromy map $M$) restricts to a map
\begin{align*}
W \colon \HH^{p,q}_{\trop,c}(\Xan, R) \to \HH^{p-1,q+1}_{\trop,c}(\Xan, R[\Gamma]).
\end{align*}
\end{satz}

\begin{proof}
By Proposition \ref{prop comparison coefficients}, it is sufficient to 
prove this theorem for $\Trop_{\varphi}(X)$. 
Since $\Trop_{\varphi}(X)$ is an integral $\Gamma$-affine tropical variety, 
this follows from Proposition \ref{prop wave restricts}. 
\end{proof}

Mikhalkin and Zharkov conjectured that for a smooth tropical variety 
$X$, the iterated wave operator 
\begin{align*}
W^{p-q} \colon \HH^{p,q}_{\trop}(X, \R) \to \HH^{q,p}_{\trop}(X, \R)
\end{align*}
is an isomorphism for all $p \geq q$ \cite[Conjecture 5.3]{MikZhar}. 

Liu conjectured that if $K$ is such that the residue field $\tilde K$ is the algebraic closure of a finite 
field and $X$ is smooth and proper, 
then the iterated monodromy operator 
\begin{align*}
M^{p-q} \colon \HH^{p,q}(X) \to \HH^{q,p}(X)
\end{align*}
is an isomorphism for all $p \geq q$ \cite[Conjecture 5.2]{Liu2}. 

As a consequence of Theorem \ref{kor wave mondromy analytic}
we can tie together both of these conjectures. 

\begin{prop}
Let $X$ be a proper variety.
Assume there exists a global admissible family of tropicalizations $\Scal$
for $X$ such that for all $\varphi \in \Scal_{\map}$ the tropical variety 
$\Trop_{\varphi}(X)$ satisfies Mikhalkin's and Zharkov's conjecture. 
Then $X$ satisfies Liu's conjecture. 
\end{prop}
\begin{proof}
This follows directly from Theorem \ref{kor wave mondromy analytic} and 
Theorem \ref{thm cohomology limit}.
\end{proof}

\section{Non-trivial classes} \label{sect non-trivial classes}

In this section we (partially) compute tropical cohomology with coefficients in three examples: 
Curves of good reduction, toric varieties and Mumford curves. 

For the first theorem assume that the value group $\Gamma$ of $K$ is a subring of $\R$. 
We denote by $\log \vert \Ocal_X^\times \vert$ the sheaf of real valued functions on $\Xan$ 
that are locally of the form $\log \vert f \vert$ for an invertible function $f$ on $X$.

\begin{satz} \label{thm Pic0}
Let $K$ be algebraically closed and $X$ be a smooth projective curve of good reduction. 
Then there exists an injective morphism 
\begin{align*}
\Pic^0(\tilde X) \to \HH^{1,1}(X, \Z),
\end{align*}
where $\Pic^0(\tilde X)$ denotes the group of degree $0$ line bundels on the reduction $\tilde X$ of $X$. 
\end{satz}

\begin{proof}
We have the following exact sequence
\begin{align} \label{eq exp sheaves}
0 \to \Gamma \to \log \vert \Ocal_X^\times \vert \to \FS^1_\Z \to 0,
\end{align}
which is a non-archimedean version of a well-known exponential sequence 
from tropical geometry \cite[Definition 4.1]{MikZharII}. 
This induces the following exact sequence in cohomology groups

\begin{align} \label{eq exp groups}
0 \to \HH^{1,0}_{\trop}(\Xan, \Z) \to  \HH^{0,1}_{\trop}(\Xan, \Gamma) \to 
\HH^1_{\trop}(\Xan, \log \vert \Ocal_X^\times \vert) \to \HH^{1,1}_{\trop}(\Xan, \Z) \to 0. 
\end{align}
In particular, since $X$ is a curve of good reduction, $\Xan$ is contractible and 
$\HH^{0,1}_{\trop}(\Xan, \Gamma) = \HH^1(\Xan, \Gamma) = 0$. 
Hence we have that
\begin{align*}
\HH^1(\Xan, \log \vert \Ocal_X^\times \vert) \to \HH^{1,1}(\Xan, \Z).
\end{align*}
is an isomorphism.
Therefore it is sufficient to prove that there exists 
an injective morphism $\Pic^0(\tilde X) \to \HH^{1}(\Xan, \log \vert \Ocal_X^\times \vert)$. 
We have an exact sequence
\begin{align} \label{eq Thuillier sheaves}
0 \to \log \vert \Ocal_X^\times \vert \to \HS_\Z \to \iota_* \Pic^0(\tilde X) \to 0
\end{align}
where $\HS_\Z$ is the sheaf of real valued functions on $\Xan$ that locally factor as the retraction
to a skeleton composed with a piecewise linear function with integer slopes and values in $\Gamma$ on the edges 
of said skeleton. 
This is sequence is the integral version of \cite[Lemme 2.3.22]{Thuillier}. 
Further, Thuillier showed that every harmonic function on a compact Berkovich analytic 
space is constant \cite[Proposition 2.3.13]{Thuillier}. 
Thus we obtain the following long exact sequence: 
\begin{align} \label{eq Thuillier groups}
0 \to \Gamma \to \Gamma \to \Pic^0(\tilde X) \to \HH^1(\Xan, \log \vert \Ocal_X^\times \vert).
\end{align}
This shows the existence of an injective morphism 
$\Pic^0(\tilde X) \to \HH^1(\Xan, \log \vert \Ocal_X^\times \vert)$.
\end{proof}

\begin{bem}
Since the Picard group of a smooth projective curve of positive genus over an algebraically closed field
contains torsion, Theorem \ref{thm Pic0} implies that $\HH^{1,1}(\Xan, \Z)$ can contain torsion.
In other word the map 
$\HH^{p,q}(\Xan, \Z) \to \HH^{p,q}(\Xan, \R) = \HH^{p,q}(\Xan, \Z) \otimes \R$ 
need not be injective. 

It is very possible that one can drop the assumption for $K$ to be algebraically closed in Theorem \ref{thm Pic0}. 
\end{bem}

\begin{satz} \label{thm trop injective}
Let $\varphi \colon X \to Y$ be a closed embedding of $X$ into a toric variety $Y$. 
Assume that $\Trop_{\varphi}(X)$ is a smooth tropical variety. 
Then  
$\trop^* \colon \HH^{p,q}(\Trop_{\varphi}(X)) \to \HH^{p,q}(\Xan)$ and  
$\trop^* \colon \HH^{p,q}_c(\Trop_{\varphi}(X)) \to \HH^{p,q}_c(\Xan)$ are injective. 
\end{satz}
\begin{proof}
By \cite[Theorem 4.33]{JSS}, since $\Trop_{\varphi}(X)$ is smooth
there is a perfect pairing 
\begin{align*}
\HH^{p,q}(\Trop_{\varphi}(X)) \times \HH^{n-p,n-q}_c(\Trop_{\varphi}(X)) \to \R
\end{align*}
induced by the wedge product and integration of superforms. 
Thus given a $d''$-closed $\alpha$ in $\AS^{p,q}(\Trop_{\varphi}(X))$ 
whose class $[\alpha] \in \HH^{p,q}(\Trop_{\varphi}(X))$ is non-trivial, 
there exists $[\beta] \in \HH_c^{n-p,n-q}(\Trop_{\varphi}(X))$ such that 
$\int_{\Trop_{\varphi}(X)} \alpha \wedge \beta \neq 0$. 
Thus we have 
\begin{align*}
\int_{\Xan} \trop^*_\varphi \alpha \wedge \trop^*_\varphi \beta \neq 0.
\end{align*}
Since integration and the wedge product are well defined on cohomology 
this means that $[\trop^*_\varphi \alpha \wedge \trop^*_\varphi \beta]$ and consequently 
$[\trop^*_\varphi \alpha]$ is not trivial.
The argument for $\alpha \in \HH^{p,q}_c(\Trop_{\varphi}(X))$ works the same except 
$[\beta] \in \HH^{n-p,n-q}(\Trop_{\varphi}(X))$. 
\end{proof}  

\begin{Ex}
Let $Y_\Sigma$ be a smooth toric variety. 
Then $Y_\Sigma$ is locally isomorphic to $\A^n$ and hence 
$\Trop(Y)$ is locally isomorphic to $\Trop(\A^n)$ and hence is a smooth tropical variety. 
Thus 
\begin{align*}
\trop^* \colon \HH^{p,q}(\Trop(Y)) \to \HH^{p,q}(Y^{\an})
\end{align*}
is injective by Theorem \ref{thm trop injective}. 
Let $Y_\Sigma(\C)$ be the complex toric variety associated with $\Sigma$. 
Then $\HH^{p,q}_{\Hodge}(Y_\Sigma(\C)) \cong \HH^{p,q}(\Trop(Y_\Sigma), \C)$ \cite[Corollary 2]{IKMZ}. 
In particular we have $\dim_{\R} \HH^{p,q}(Y_{\Sigma}) \geq \dim_{\C} \HH^{p,q}_{\Hodge}(Y_\Sigma(\C))$. 
One may figure out the latter in terms of $\Sigma$ using \cite[Section 5.2]{Fulton} or with the 
help of a computer and in terms of the polytope of $Y$ using the package cellularSheaves for polymake \cite{KSW}. 
Note that $\HH^{p,q}_{\Hodge}(Y(\C)) = 0$ if $p \neq q$ by \cite[Corollary 12.7]{Danilov}. 
\end{Ex}

\begin{Ex} \label{example Mumford}
Let $K$ be algebraically closed, $X$ be a smooth projective curve of genus $g$ and $\varphi \colon X \to Y$ be 
a closed embedding of $X$ into a toric variety such that $\Trop_{\varphi}(X)$ is a smooth tropical variety
(this exists if and only if $X$ is a Mumford curve by \cite{JellSmooth}). 
Then 
\begin{align*}
\HH^{p,q}_{\trop}(\Trop_{\varphi}(X), R) \to \HH^{p,q}_{\trop}(\Xan, R)
\end{align*}
is an isomorphism. 
In particular we have $\HH^{0,0}_{\trop}(\Xan, R) \cong \HH^{1,1}_{\trop}(\Xan, R) \cong R$ and 
$\HH^{1,0}_{\trop}(\Xan, R) \cong \HH^{0,1}_{\trop}(\Xan, R) \cong R^g$. 
\end{Ex}
\begin{proof}
Assume that $\Trop_{\varphi}(X)$ is smooth. 
Then $\trop_{\varphi}$ is a homeomorphism from a skeleton of 
$\Xan$ onto $\Trop_{\varphi}(X)$ \cite[Theorem 5.7]{Jell}.
Using comparison with singular cohomology we obtain $\HH^{0,0}_{\trop}(\Trop_{\varphi}(X), R) = R$
and $\HH^{0,1}_{\trop}(\Trop_{\varphi}(X), R) = R^g$. 
Using duality with coefficients in $R$ as proven in \cite[Theorem 5.3]{JRS} and comparison 
with singular homology,
we also obtain $\HH^{1,1}_{\trop}(\Xan, R) = R$ and $\HH^{1,0}_{\trop}(\Trop_{\varphi}(X), R) = R^g$. 
One immediately verifies all transition maps induced 
by refinements in the family $\tor_{\Smooth}$ defined in Example \ref{ex smooth} 
are isomorphisms and hence the claim follows from Theorem \ref{thm cohomology limit}. 
\end{proof}

\section{Open questions} \label{sect open questions}

In this section, we let $X$ be a variety over $K$. 

When $X$ is smooth, Liu constructed cycles class maps, meaning maps 
$\cyc_k \colon \CH(X)^k \to \HH^{k,k}(\Xan)$ 
that are compatible with the product structure on both sides and have
the expected integration property \cite{Liu}. 

\begin{question}
What is the image of $\cyc_k$?
\end{question}

In light of the tropical Hodge conjecture and Corollary \ref{kor wave mondromy analytic}, 
one might conjecture 
that the image of $\CH(X)_\Q$ is $\HH^{k,k}(\Xan, \Q) \cap  \ker(M)$. 
One might start with the case $k = \dim X - 1$. 
Here one knows the answer tropically \cite{JRS}, 
but the non-archimedean analogue is not a direct consequence. 

The following question was asked by a referee and 
the author thinks it is worthwhile to include it 
here along with a partial answer. 

\begin{question}
Is there an analogue of Theorem \ref{thm Pic0} when $X$ has semistable reduction?
\end{question}

Let $X$ be a smooth projective curve, let $\Xcal_s$ be the special fiber of a strictly semistable model of $X$ and 
let $C_1,\dots,C_n$ be the irreducible components of $\Xcal_s$. 
Then we can construct a map $\Pic^0(\Xcal_s) \to \HH^{1,1}(\Xan, \Z)$ 
by the composition
\begin{align*}
\Pic^0(\Xcal_s) \to \bigoplus_{i = 1}^n \Pic^0(C_i) \to \HH^{1,1}(\Xan, \log \vert \Ocal_X^\times \vert) 
\to \HH^{1,1}(\Xan, \Z). 
\end{align*}
Here the first map is pullback along normalization, and the second and third maps 
are induced by (\ref{eq Thuillier groups}) and (\ref{eq exp groups}), 
which remain valid when replacing $\Pic^0(\tilde X)$ with $\bigoplus_{i = 1}^n \Pic^0(C_i)$. 
It is well known that the first map need not be injective, hence the composition need
not be injective. 
Whether the map
\begin{align*}
\bigoplus_{i = 1}^n \Pic^0(C_i) \to \HH^{1,1}(\Xan, \log \vert \Ocal_X^\times \vert) 
\to \HH^{1,1}(\Xan, \Z)
\end{align*}
is injective is unclear to the author. 
The map  $\HH^{1,1}(\Xan, \log \vert \Ocal_X^\times \vert) \to \HH^{1,1}(\Xan, \Z)$
will not be injective when $\Xan$ is not contractible, 
since the map $\HH^{1,0}(\Xan, \Z) \to \HH^{0,1}(\Xan, \Gamma)$ from $(\ref{eq exp groups})$
will not be surjective. 
But that of course does not imply that the composition can not be injective.

\begin{question}
Does there exists a toric variety $Y$ and a closed embedding 
$\varphi \colon X \to Y$
such that 
\begin{align*}
\trop^* \colon \HH^{p,q}(\Trop_{\varphi}(X)) \to \HH^{p,q}(\Xan) \text{ and }
\trop^* \colon \HH^{p,q}_c(\Trop_{\varphi}(X)) \to \HH^{p,q}_c(\Xan)
\end{align*}
are isomorphisms?
\end{question}
The statement for $\HH^{p,q}_c(\Xan)$ is implied by the finite dimensionality of $\HH^{p,q}_c(\Xan)$
via Theorem \ref{thm cohomology limit}. 
It is in fact equivalent to the finite dimensionality of $\HH^{p,q}(\Xan)$ if one knew that 
$\HH^{p,q}_c(\Trop_{\varphi}(X))$ is always finite dimensional, though 
the author is not aware of such a result 
(without regularity assumptions on $\Trop_{\varphi}(X)$).

Other questions related to this concern smoothness of the tropical variety. 

\begin{question}
Let $\varphi \colon X \to Y$ be a closed embedding of $X$ into a toric variety $Y$ 
such that $\Trop_{\varphi}(X)$ is smooth. 
Are then
\begin{align*}
\trop^* \colon \HH^{p,q}(\Trop_{\varphi}(X)) \to \HH^{p,q}(\Xan) \text{ and } 
\trop^* \colon \HH^{p,q}_c(\Trop_{\varphi}(X)) \to \HH^{p,q}_c(\Xan)
\end{align*}
isomorphisms? 
\end{question}

This is certainly a natural question and ``optimistically expected'' to be true by Shaw \cite[p.3]{ShawSimons}.
We now know it holds for curves, as we showed in Example \ref{example Mumford}, but even the case $X = Y$ is 
open in dimension $\geq 2$.

\begin{question}
Let $\varphi \colon X \to Y$ be a closed embedding of $X$ into a toric variety $Y$ 
such that $\Trop_{\varphi}(X)$ is smooth. 
Does the diagram
\begin{align*}
\begin{xy}
\xymatrix{
Z^k(X)\ar[d]_{\trop}  \ar[rr] && \CH(X) \ar[rr]^{\cyc_k}  && \HH^{k,k}(\Xan)  \\
Z^k(\Trop(\varphi(X)) \ar[rr]^{Z \mapsto \int_{Z}} && 
\HH^{n-k,n-k}(\Trop_{\varphi}(X))^* \ar[rr]^{\PD^{-1}} && \HH^{k,k}(\Trop_{\varphi}(X)) \ar[u]_{\trop^*}
}
\end{xy} 
\end{align*}
commute? Here $\PD$ denote the Poincar\'e duality isomorphism on tropical varieties \cite{JSS}
and $\cyc_k$ denotes Liu's cycles class map \cite{Liu}. 
\end{question}

Let us finish with the remark that the author does not 
know of any variety $X$ with $\dim(X) \geq 2$ and 
any $0 < p \leq \dim(X)$ and $0 < q \leq \dim(X)$ with $(p,q) \neq (1,1)$ 
where we know $\dim_\R \HH^{p,q}(\Xan)$. 
(No, not even $\HH^{2,2}(\mathbb{P}^{2, \an})$ or $\HH^{2,2}(\A^{2, \an})$.)

\bibliographystyle{alpha}
\def\cprime{$'$}

\end{document}